\documentclass[12pt,reqno]{article}

\usepackage[T1]{fontenc}
\usepackage{amsmath}
\usepackage{amsthm}
\usepackage{amscd}
\usepackage{amssymb}
\usepackage{enumerate}
\usepackage{latexsym}
\usepackage{xypic}

\usepackage{epsfig}

\usepackage{graphicx}
\def\epsilon{\varepsilon}

\newcommand{\Out}{\mathrm{Out}}

\newcommand{\inv}{^{-1}}

\newcommand{\hdim}{\mathrm{Hdim}}
\newcommand{\gap}{g}

\newcommand{\FN}{F_N}   

\newcommand{\barCVN}{\bar{\mathrm{CV}}_N}   

\newcommand{\CA}{{\cal A}}

\newcommand{\Lleg}{L_\tau}
\newcommand{\Latt}{L_\Phi}

\newcommand{\CQ}{{\cal Q}}

\newcommand{\TPhiinvobs}{\widehat T_{\Phi\inv}^{\text{\scriptsize obs}}}

\newcommand{\Rauzy}{\mathcal{R}_\sigma}

\renewcommand{\phi}{\varphi}

\newcommand{\R}{\mathbb R}

\newcommand{\N}{\mathbb N}

\def\bar{\overline}
\def\tilde{\widetilde}
\def\hat{\widehat}

\newtheorem{thm}{Theorem}[section]
\newtheorem*{thm*}{Theorem}
\newtheorem*{thmhausdim}{Theorem~\ref{thm:hausdim}}
\newtheorem{cor}[thm]{Corollary}
\newtheorem{lem}[thm]{Lemma}
\newtheorem{prop}[thm]{Proposition}

\theoremstyle{definition}

\newtheorem*{defn*}{Definition}

\newtheorem*{rem*}{Remark}

\theoremstyle{remark}

\numberwithin{equation}{section}

\begin{document}

\title{Fractal trees for irreducible automorphisms of free groups}

\author{Thierry Coulbois}

\date{\today }

\maketitle

\begin{abstract}
  The self-similar structure of the attracting subshift of a primitive
  substitution is carried over to the limit set of the repelling tree
  in the boundary of Outer Space of the corresponding irreducible
  outer automorphism of a free group. Thus, this repelling tree is
  self-similar (in the sense of graph directed constructions).  Its
  Hausdorff dimension is computed. This reveals the fractal nature of
  the attracting tree in the boundary of Outer Space of an irreducible
  outer automorphism of a free group.
\end{abstract}

\tableofcontents

\newpage

Throughout this article, $\FN$ denotes the free group of finite rank $N\geq 2$.

An $\R$-tree $(T,d)$ is an arcwise connected metric space such that
two points $P$ and $Q$ are connected by a unique arc and this arc is
isometric to the real segment $[0,d(P,Q)]$. An $\R$-tree is usually
regarded as a $1$-dimensionnal object. And, indeed, if $T$ is a
non-trivial $\R$-tree with a minimal action of $\FN$ by isometries,
then $T$ is a countable union of arcs and thus has Hausdorff dimension
$1$. 

Surprisingly, we exhibit in this article $\R$-trees $T$ in the boundary
of M.~Culler and K.~Vogtmann's Outer Space $\barCVN$ (which is made of
$\R$-trees with minimal, very-small action of $\FN$ by isometries),
such that the Hausdorff dimension of their metric completion $\bar T$
is strictly bigger than $1$.

More precisely, we prove that, for an irreducible (with irreducible
powers) outer automorphism $\Phi$ of $\FN$, the metric completion
$\bar T_\Phi$ of the attracting tree $T_\Phi$ in the boundary of Outer
Space has Hausdorff dimension
\[
\hdim(\bar T_\Phi)\geq\max(1;\frac{\ln\lambda_{\Phi\inv}}{\ln\lambda_\Phi})
\]
where $\lambda_\Phi$ and $\lambda_{\Phi\inv}$ are the expansion
factors of $\Phi$ and $\Phi\inv$ respectively. We insist that these
two expansion factors may be distinct leading to a Hausdorff dimension
strictly bigger than $1$.

This lower bound on the Hausdroff dimension is achieved by computing
the exact Hausdorff dimension of a subset of the metric completion:
the limit set $\Omega$. This is the subset of $\bar T_\Phi$ where the
dynamic of $\Phi$, as given by the repelling lamination concentrates.

\medskip

For an irreducible (with irreducible powers) outer automorphism
$\Phi$ of the free group $\FN$, M.~Bestvina, M.~Feighn and M.~Handel
(\cite{bfh-lam}) defined the attracting lamination $\Latt$. By choosing a
basis $\CA$ of $\FN$, the lamination $\Latt$ can be viewed as a
symbolic dynamical system (indeed a subshift of the shift on
bi-infinite reduced words in $\CA^{\pm 1}$) as explained in
\cite{chl1-I} and briefly recalled in Section~\ref{sec:laminations}.

The attracting lamination $\Latt$ is best described if we choose a
train-track representative $\tau=(\Gamma,*,\pi,f)$ of $\Phi$, where
$\Gamma$ is a finite graph with base point $*$, $\pi$ is a marking
isomorphism between $\FN$ and the fundamental group $\pi_1(\Gamma,*)$
and $f$ is a homotopy equivalence inducing $\Phi$ via
$\pi$. M.~Bestvina and M.~Handel (\cite{bh-traintrack}) defined
train-track representatives and proved that they always exist for
irreducible (with irreducible powers) outer automorphisms of $\FN$
(see Sections~\ref{sec:autom-topol-repr} and
\ref{sec:train-track-repr}). The lamination $\Latt$ is a closed set of
bi-infinite paths in the universal cover $\tilde\Gamma$ of $\Gamma$,
and it is invariant under application of any lift $\tilde f$ of $f$ to
$\tilde\Gamma$ (see Section~\ref{subsec:attlam}).

Using the chart given by the train-track representative $\tau$ to
describe the attracting lamination $\Latt$ we get in
Proposition~\ref{prop:lamfrac} a self-similar decomposition of $\Latt$
into finitely many cylinders. Self-similarity is here to be understood
in the sense of graph directed constructions as introduced by
\cite{mw} which is a generalisation of iterated function systems. We
refer to \cite{edgar} for introduction and background on this topic.

The self-similar structure of the attracting lamination is wellknown
to symbolic dynamists and a key tool to deal with it is the
prefix-suffix automaton (see Section~\ref{sec:prefix}).

In this article we carry over this self-similar decompostion of the
attracting lamination, which is partly folklore,
to the limit set of the repelling tree $T_{\Phi\inv}$ of $\Phi$ in the
boundary of Culler-Vogtman Outer Space.  We refer to K.~Vogtman's
survey \cite{vogt-survey} for background on Outer Space.

A construction of the repelling tree $T_{\Phi\inv}$ of $\Phi$ can be found in
\cite{gjll}. It is an $\R$-tree with a very small, minimal action of
$\FN$ by isometries with dense orbits. It comes with a contracting
homothety $H$ associated to the choice of a representative
automorphism $\phi$ of the outer class $\Phi$. The ratio of $H$ is
$\frac{1}{\lambda_{\Phi\inv}}$, where $\lambda_{\Phi\inv}$ is the
expansion factor of $\Phi\inv$ (see Section~\ref{sec:repel-tree}).

From~\cite{ll-north-south,ll-north-south,chl1-II} (see
Sections~\ref{subsec:Q}, \ref{subsec:Q2} and
\ref{sec:dual-attr-lamin}), there exists a continuous map $\CQ^2$ that
maps the attracting lamination $\Latt$ into the metric completion
$\bar T_{\Phi\inv}$ of the repelling tree $T_{\Phi\inv}$. The
self-similar decomposition of the attracting lamination is carried
over through $\CQ^2$ to get a self-similar limit set $\Omega$ inside
$\bar T_{\Phi\inv}$. Using the ratio $\frac{1}{\lambda_{\Phi\inv}}$ of
the homothety $H$, we get the main result of this article:

\begin{thmhausdim}
  Let $\Phi$ be an irreducible (with irreducible powers) outer
  automorphism of the free group $\FN$. Let $T_{\Phi\inv}$ be the
  repelling tree of $\Phi$.
  
The limit set $\Omega\subseteq\bar T_{\Phi\inv}$ has Hausdorff dimension
\[
\delta_{\Phi\inv}=\hdim(\Omega)=\frac{\ln\lambda_{\Phi}}{\ln\lambda_{\Phi\inv}}
\]
where $\lambda_\Phi$ and $\lambda_{\Phi\inv}$ are the expansion
factors of $\Phi$ and $\Phi\inv$ respectively.
\end{thmhausdim}

Knowing $\delta_{\Phi\inv}$ we can use the Hausdorff measure in dimension
$\delta_{\Phi\inv}$ to describe the correspondence between the unique ergodic
probability measure carried by the attracting lamination and the
metric of the $\R$-tree $T_{\Phi\inv}$.

We insist that the expansion factors of an irreducible (with
irreducible powers) outer automorphism and its inverse are not equal
in general. Surprisingly, this leads to compact subsets  of an $\R$-tree
which can be of Hausdorff dimension strictly bigger than $1$ although
an $\R$-tree is usually regarded as a $1$-dimensional object.

During his beautiful course on the Mapping Class Group at MSRI in Fall
2007, L.~Mosher mentioned that the convex core (see \cite{guir-core})
of the product of the attracting and repelling trees of an irreducible
(with irreducible powers) parageometric automorphism should be of
Hausdorff dimension given by the ratio $\delta_{\Phi\inv}$ of the
logarithms of the expansion factors of the automorphism and its
inverse. This lead us to understand that the limit set of the
repelling tree has the Hausdorff dimension $\delta_{\Phi\inv}$.

The main difficulty in proving our Theorem is to carefully study how the
self-similar pieces of the limit set intersect. This involves
describing the points that belong to more than one piece and proving
that their prefix-suffix representations are periodic.

Finally in Section~\ref{sec:examples} we describe two classical
examples and detail the shape of the limit sets and compact hearts.

\medskip

In the end of this introduction we want to recall two more classical
constructions which are sources of inspiration for our work.

\medskip

The above picture is very different from the situation of
pseudo-Anosov mapping classes which are a source of inspiration for
studying outer automorphims. Indeed, a pseudo-Anosov homeomorphism
$\phi$ of a hyperbolic surface and its inverse have the same expansion
factor.  Recall that the mapping class $\Phi$ of an homeomorphim
$\phi$ of a surface $S$ induces an outer automorphism of the
fundamental group of the surface. And if the surface has non-trivial
boundary, its fundamental group is a free group. The pseudo-Anosov
homeomorphism $\phi$ comes with an unstable foliation
$\mathcal{F}_\phi$ on the hyperbolic surface $S$. Tightening this
foliation we get the unstable geodesic lamination $\mathfrak{L}_\Phi$
of the mapping class $\Phi$ of which the attracting lamination of
$\Phi$ is the algebraic version.  Under iterations of $\phi$, any
closed curve converges to the unstable geodesic lamination.

The mapping class $\Phi$ also acts on Teichmüller space and its
boundary and has a repelling fixed point which can be described as an
$\R$-tree $T_{\Phi\inv}$ with small action of the fundamental group of the
surface. Geometrically the tree $T_{\Phi\inv}$ is transverse to the lift of the
unstable geodesic lamination to the universal cover of the surface.

The limit set $\Omega$ of $T_{\Phi\inv}$ is equal to $T_{\Phi\inv}$
and is a countable union of intervals.  Thus its Hausdorff dimension
is $1$ which is consistent with our Theorem. 

Alternatively, following W.~Thurston \cite{flp}, the pseudo-Anosov
mapping class $\Phi$ fixes a train-track on the surface $S$. This
train-track carries the unstable foliation. The compact sets
$\Omega_{\tilde e}$ (see section~\ref{sec:fractalomega}) for this
train-track are intervals transverse to the foliation. The first
return map $T$ along the unstable foliation, on the union of these
intervals is an interval exchange transformation.

\medskip

Let us now review the above description in the case of an irreducible
(with irreducible powers) outer automorphism $\Phi$ represented by a
substitution $\sigma$. We refer to N.~Pytheas Fogg \cite{fogg} for
background and results on symbolic dynamics.

Let $\Phi$ be an outer automorphism of $\FN$ which admits a basis
$\CA$ of $\FN$ and a representative $\sigma$ which is a substitution
(that is to say, only positive letters appear in the images of the
elements of $\CA$). In this case we rather regard $\sigma$ as an
homomorphism of the free monoid on the alphabet $\CA$. 

Under iterations of $\sigma$, any letter $a\in\CA$ converges to the
attracting subshift $\Sigma_\sigma$. This is the subshift of the full
shift on bi-infinite words in $\CA$ which consists of bi-infinite
words whose finite factors are factors of images of $a$ under
iterations of $\sigma$. Considering the shift map $S$ we get a
symbolic dynamic system $(\Sigma_\sigma,S)$.

This attracting subshift $\Sigma_\sigma$ is the (symbolic)
attracting lamination $\Latt$ of the irreducible (with irreducible
powers) outer automorphism $\Phi$ (more precisely it is half of
$\Latt$ as we fixed, as a convention, that laminations are invariant
by taking inverses). The self-similar decomposition of the attracting
subshift occurs in this case in the basis $\CA$ which is a train-track
for $\Phi$ and is well-known to dynamists.

If in addition, the substitution $\sigma$ satisfies the
arithmetic-type Pisot condition, then the dynamical system
$(\Sigma_\sigma, S)$ has a geometric interpretation as a Rauzy
fractal $\Rauzy$.

The Rauzy fractal $\Rauzy$ is a compact subset of
$\R^{N-1}$. The Rauzy fractal is graphically striking when $N=3$ in
which case it is a compact subset of the plane. The Rauzy fractal
comes with a piecewise exchange $T$. The dynamical system $(\Rauzy,T)$
is semi-conjugated with the attracting subshift $(\Sigma_\sigma,S)$.

Indeed, V.~Canterini and A.~Siegel \cite{cs} defined a map $R$ from
the attracting shif $\Sigma_\sigma$ onto the Rauzy fractal: A
bi-infinite word $Z$ in the attracting subshift $\Sigma_\sigma$
corresponds to the trajectory of exactly one point, $R(Z)$ of
$\Rauzy$. The map $R$ is continuous and onto and therefore $\Rauzy$ is
a geometric representation of the dynamic of the attracting subshift.

The map $R$ factors through the map $\CQ^2$ which means that the Rauzy
fractal is a quotient of the compact limit set $\Omega_\CA$ of the repelling
tree $T_{\Phi\inv}$ of $\Phi$.

The self-similar decomposition of the attracting subshift
$\Sigma_\sigma$, described by the prefix-suffix automaton, is carried
over by the continuous map $R$ to $\Rauzy$.  The self-similar
decomposition of the Rauzy fractal $\Rauzy$ obtained is the same as
the self-similar decomposition of $\Omega_\CA$ described in
Proposition~\ref{prop:heartfrac}.

However, we note that the self-similar decomposition of the Rauzy
fractal does not lead directly to a meaningful Hausdorff dimension
because intersections between pieces may not be neglectable: the map
$R$ is non-injective in a ``Hausdorff-dimension'' essential way.

\noindent\textbf{Acknowledgement:} This research started at the MSRI,
while I was attending the special semester on Geometric Group Theory
organised by M.~Bestvina, J.~McCammond, M.~Sageev, and K.~Vogtmann. I
wish to thank them for the beautiful opportunity they offered us.

If I had not been in California at that time, this paper would
certainly be a joint paper with my favorite co-authors: A.~Hilion and
M.~Lustig.

This work greatly benefited from the insight of X.~Bressaud's PhD
student, Y.~Jullian, who helped me a lot to handle graph directed
constructions and self-similarity tools and concepts.

\section{Laminations and Automorphisms}\label{sec:lamaut}

\subsection{Laminations}\label{sec:laminations}

The free group $\FN$ is Gromov-hyperbolic and has a well defined
boundary at infinity $\partial\FN$, which is a topological space,
indeed a Cantor set.

The action of $\FN$ on its boundary is by homeomorphisms.

The \textbf{double boundary} of $\FN$ is
\[
\partial^2\FN=(\partial\FN)^2\smallsetminus\Delta
\]
where $\Delta$ is the diagonal. An element of $\partial^2\FN$ is a \textbf{line}.

A \textbf{lamination} (in its algebraic setting) is a closed,
$\FN$-invariant, flip-invariant subset of $\partial^2\FN$ (where the
flip is the map exchanging the two coordinates of a line). The
elements of a lamination are called \textbf{leaves}.

We refer the reader to \cite{chl1-I} where laminations for free groups
 are defined and different equivalent approaches are exposed with
 care.

\subsection{Charts and Cylinders}\label{sec:charts-cylinders}

To give a geometric interpretation of the boundary, of leaves and
of laminations we introduce charts.

Let $\Gamma$ be a finite graph, with basepoint $*$ and
$\pi:\FN\to\pi_1(\Gamma,*)$ a \textbf{marking} isomorphism.
We say that $(\Gamma,*,\pi)$ is a \textbf{chart} for $\FN$.

Assigning a positive length to each edge in $\Gamma$ (e.g. $1$ to each
edge) defines a path metric of the universal cover $\tilde\Gamma$.
For such a metric, $\tilde\Gamma$ is a 0-hyperbolic space, indeed a
tree, and it has a boundary at infinity $\partial\tilde\Gamma$ which
is simply the space of ends. Points of the boundary
$\partial\tilde\Gamma$ can be seen as infinite geodesic paths starting
from a fixed lift $\tilde *$ of the base point $*$.
 
The action of $\FN$ on $\tilde\Gamma$ by deck transformations through
the marking $\pi$ is by isometries. We denote by
$\partial\pi:\partial\FN\to\partial\tilde\Gamma$ the canonical
homeomorphism between the boundaries at infinity.

Through $\partial\pi$ there is a canonical correspondence, which
associates to a line $(X,Y)\in\partial^2\FN$ the geodesic bi-infinite
oriented arc of $\tilde\Gamma$ $[\partial\pi(X),\partial\pi(Y)]$ joining the
points at infinity $\partial\pi(X)$ and $\partial\pi(Y)$. We say that
this bi-infinite geodesic path is the \textbf{geometric realisation} of
the line $(X,Y)$.

For a finite oriented geodesic arc $\gamma$ in $\tilde\Gamma$, the
\textbf{cylinder} of $\gamma$ $C_\Gamma(\gamma)$ is the set of
lines whose geometric realisations contain $\gamma$.  

Cylinders are closed-open sets and they form a basis of the topology
of $\partial^2\FN$. An element $u$ of $\FN$ translates by left
multiplication the cylinder $C_\Gamma(\gamma)$ to
$uC_\Gamma(\gamma)=C_\Gamma(u\gamma)$.

\subsection{Automorphisms and topological representatives}\label{sec:autom-topol-repr}

Let $\phi$ be an automorphism of $\FN$. It extends canonically to an
homeomorphism $\partial\phi:\partial\FN\to\partial\FN$ and also
induces an homeomorphism, $\partial^2\phi$ of $\partial^2\FN$.

For example, the inner automorphism $i_u:x\mapsto uxu\inv$, defined by
the conjugation by the element $u$ of $\FN$, acts on $\partial\FN$ as
(the left multiplication by) $u$.

If $L$ is a lamination $\phi(L)=\{(\partial\phi(X),\partial\phi(Y))\
|\ (X,Y)\in L\}$ is also a lamination. As a lamination is invariant
under the action of $\FN$, inner automorphisms act trivially on the
set of laminations, and we get an action of the outer automorphism
group $\Out(\FN)$ on the set of laminations. We consistently denote by
$\Phi(L)=\phi(L)$ the image of $L$ by the outer class $\Phi$ of
$\phi$.

If $(\Gamma,*,\pi)$ is a chart for $\FN$ as in the previous section, a
\textbf{topological representative} of the outer automorphism $\Phi$
is a continuous map $f:\Gamma\to\Gamma$ which sends vertices to
vertices, edges to finite reduced paths, and which is a homotopy
equivalence inducing $\Phi$ through the marking $\pi$.  A lift $\tilde
f:\tilde\Gamma\to\tilde\Gamma$ of $f$ to the universal cover
$\tilde\Gamma$ of $\Gamma$ is a \textbf{topological representative} of
the automorphism $\phi\in\Phi$ if the following condition holds:
\[
\forall P\in\tilde\Gamma, \forall u\in\FN, \tilde f(uP)=\phi(u)\tilde f(P).
\]
If $\psi=i_u\circ\phi$ is another automorphism in the outer class
$\Phi$, then $\tilde f'=u\tilde f$ is a topological representative of $\psi$:
Lifts of $f$ are in one-to-one correspondence with automorphisms in
the outer class $\Phi$.

For any lift $\tilde f$ of the homotopy equivalence $f$ of $\Gamma$,
$\tilde f$ is a quasi-isometry of $\tilde\Gamma$ and extends to an
homeomorphism, $\partial\tilde f$, of the boundary at infinity
$\partial\tilde\Gamma$.

If
$\tilde f$ is a topological representative of the automorphism $\phi$
then the following diagram  commutes:
\[
\xymatrix{ \partial F_N \ar@{->}[r]^{\partial\phi}_\cong
    \ar@{->}[d]^{\partial\pi}_\cong & \partial\FN
    \ar@{->}[d]^{\partial\pi}_\cong\\ 
    \partial\tilde\Gamma
    \ar@{->}[r]^{\partial\tilde f}_\cong & \partial\tilde\Gamma }
\]

For a subset $C$ of $\partial^2\FN$ (e.g. a cylinder
$C_\Gamma(\gamma)$) we abuse of notations and write: $\phi(C)=\tilde
f(C)=\{(\partial\phi(X),\partial\phi(Y))\ |\ (X,Y)\in C\}$.  We note
that the homeomorphism $\partial^2\phi$ maps a cylinder
$C_\Gamma(\gamma)$ to a closed-open set of $\partial^2\FN$ which is a
finite union of cylinders but may fail to be a cylinder.

\subsection{Train-track representatives and legal lamination}\label{sec:train-track-repr}

A \textbf{train-track representative} $\tau=(\Gamma,*,\pi,f)$ of the
outer automorphism $\Phi$ of $\FN$ is a chart $(\Gamma,*,\pi)$
together with a topological representative $f$ of $\Phi$ such that for
all integer $n\geq 1$, $f^n$ is locally injective on each edge of
$\Gamma$. 

The lift $\tilde f$ of $f$ which is the topological representative of the
automorphism $\phi\in\Phi$ is a \textbf{train-track representative} for $\phi$.

The train-track $\tau$ is \textbf{irreducible} if it
contains no vertices of valence $1$ or $2$, and if $\Gamma$ contains no
non-trivial $f$-invariant subgraph.

An outer automorphism $\Phi$ is \textbf{irreducible} (with
irreducible powers) if for each $n$, $\Phi^n$ does not fix a conjugacy
class of a free factor. M.~Bestvina and M.~Handel \cite{bh-traintrack} proved that
irreducible (with irreducible powers) outer automorphisms always have
an irreducible train-track representative.

A geodesic path $\gamma$ in $\tilde\Gamma$ (finite, infinite or
bi-infinite) is \textbf{legal} if for all $n\geq 1$, the restriction
of $\tilde f^n$ to $\gamma$ is injective (this does not depend on the
choice of a particular lift $\tilde f$ of $f$). In particular, from
the definition, every 1-edge path is legal. A line
$(X,Y)\in\partial^2\FN$ is \textbf{legal} if its geometric realisation
$[\partial\pi(X);\partial\pi(Y)]$ is a legal bi-infinite path.

The \textbf{legal lamination} $\Lleg$ of the train-track
$\tau=(\Gamma,*,\pi,f)$ is the set of legal lines. From the definitions it
is clear that
\[
\Phi(\Lleg)\subseteq\Lleg.
\]
Indeed if $\phi\in\Phi$ is an automorphism representing the outer
class $\Phi$, $\partial^2\phi$ sends any legal line to a legal line.

The \textbf{transition matrix} $M$ of the homotopy equivalence
$f$ of the graph $\Gamma$ is the square matrix of size the number
of edges of $\Gamma$, and for each pair $(e,e')$ of edges of $\Gamma$,
the entry $m_{e,e'}$ is the number of occurences of $e$ in the path
$f(e')$. We insist that occurences are positive and are counted
without taking in account orientation.

The \textbf{expansion factor} $\lambda_\Phi$ of the outer automorphism
$\Phi$ is the Perron-Frobenius eigen-value $\lambda_\Phi>1$ of the
transition matrix $M$ of the irreducible train-track representative
$\tau=(\Gamma,*,\pi,f)$ of $\Phi$. The expansion factor does not
depend on the choice of a particular irreducible train-track
representative. We denote by $(\mu_e)_e$ a positive Perron-Frobenius
eigen-vector of the transition matrix $M$. This eigen-vector is unique
up to a mulitplicative constant.

\subsection{Attracting lamination}\label{subsec:attlam}

In \cite{bfh-lam} the attracting lamination of an irreducible (with irreducible powers) outer
automorphism is defined.

Let $\tau=(\Gamma,*,\pi,f)$ be an irreducible train-track
representative of the outer automorphism $\Phi$. Fix an edge $e$ in
the universal cover $\tilde\Gamma$ of $\Gamma$.  The
\textbf{attracting lamination} $\Latt$ of $\Phi$ is the set of leaves
which are limits of sequences of translates of iterated images of $e$:
\[
\Latt=\{(X,Y)\in\partial^2\FN\ |\ \exists\epsilon=\pm 1, \exists
u_n\in\FN, (X,Y)=\lim_{n\to\infty}u_n{\tilde f}^n(e^\epsilon)\}.
\]
Where the sequence of paths $u_n{\tilde f}^n(e)$ converges to the leaf
$(X,Y)$ if the sequence of startpoints converges to $\partial\pi(X)$ and the
endpoints to $\partial\pi(Y)$ in the topological space
$\tilde\Gamma\cup\partial\tilde\Gamma$.

From the definition it is clear that $\Latt$ is closed,
$\FN$-invariant and flip-invariant, indeed a lamination.  Moreover, as
$\tilde f(u_n{\tilde f}^n(e^\epsilon))=\phi(u_n){\tilde
  f}^{n+1}(e^\epsilon)$ the attracting lamination is invariant by $\Phi$.

This definition does not depend on the particular choice of a lift
$\tilde f$ of $f$, and as $\tau$ is irreducible this does not depend
either on the choice of the edge $e$ of $\tilde\Gamma$.

As $\tau$ is a train-track representative, each path ${\tilde f}^n(e)$
is legal and thus the attracting lamination is a sublamination of the
legal lamination $\Lleg$. 

It is proven in \cite{bfh-lam} that the attracting lamination does
only depend on the irreducible (with irreducible powers) outer automorphism $\Phi$ and not on the
choice of the train-track representative $\tau$. It is proven there
that the attracting lamination $\Latt$ is minimal and thus is the
smallest sublamination of the legal lamination $\Lleg$ such that
\[
\Phi(\Latt)=\Latt.
\]

\subsection{Self-similar decomposition of the attracting lamination}\label{sec:self-simil-lam}

Although we noticed that the image of a cylinder by an automorphism is
not in general a cylinder, we describe in this section the image
of legal lines contained in a cylinder.

Let $(X,Y)\in\partial^2\FN$ be a legal line for the train-track
representative $\tau=(\Gamma,*,\pi,f)$ of the outer automorphism
$\Phi$ of $\FN$. The lift $\tilde f$ of $f$ which is a train-track
representative of the automorphism $\phi\in\Phi$ restricts to an
homeomorphism from the geometric realisation
$[\partial\pi(X);\partial\pi(Y)]$ to its image
$[\partial\pi\partial\phi(X);\partial\pi\partial\phi(Y)]$.

As the attracting lamination is made of legal lines, for any (legal)
path $\gamma$ in the universal cover $\tilde\Gamma$ of $\Gamma$ we
have
\[
\phi(C_\Gamma(\gamma)\cap\Latt)\subseteq C_\Gamma(\tilde f(\gamma))\cap\Latt.
\]

For any oriented edge $\tilde e$ in $\tilde\Gamma$ we denote by $C_{\tilde e}$ the set
\[
C_{\tilde e}=C_\Gamma(\tilde e)\cap\Latt
\]

\begin{prop}\label{prop:lamfrac}
  Let $\phi\in\Phi$ be an irreducible (with irreducible powers) automorphism of $\FN$ and $\Phi$ be its
  outer class.  Let $\tau=(\Gamma,*,\pi,f)$ be a train-track
  representative for $\Phi$ and $\tilde f$ a lift of $f$ to the
  universal cover $\tilde\Gamma$ associated to $\phi$. For any edge
  $\tilde e$ of $\tilde\Gamma$
\[
C_{\tilde e}=\biguplus_{(\tilde e',\tilde p,\tilde s)} \phi (C_{\tilde e'}),
\]
the finite disjoint union is taken over triples $(\tilde e',\tilde p,\tilde s)$ such that
$\tilde e'$ is an edge of $\tilde\Gamma$, $\tilde p.\tilde e.\tilde s$ is a reduced path in
$\tilde\Gamma$ and $\tilde f(\tilde e')=\tilde p.\tilde e.\tilde s$.

The above decomposition of $C_{\tilde e}$ does not depend on the choice of a
particular automorphism $\phi$ in the outer class $\Phi$ and of its
associated lift $\tilde f$ of $f$.
\end{prop}
\begin{proof}
  By the previous remark, any leaf in $\phi(C_{\tilde e'})$ contains the
  edge $\tilde e$ which proves that $\phi (C_{\tilde e'})\subseteq C_{\tilde e}$.

  Conversely, let $(X,Y)$ be a leaf in $C_{\tilde e}$, in particular
  it is a legal leaf in $\Latt$ and there exists a legal leaf
  $(X',Y')$ in $\Latt$ such that $\partial\phi(X')=X$ and
  $\partial\phi(Y')=Y$. The map $\tilde f$ restricts to an
  homeomorphism between the bi-infinite geodesic paths
  $[\partial\pi(X');\partial\pi(Y')]$ and
  $[\partial\pi(X);\partial\pi(Y)]$, and as the former is legal and
  the latter contains the edge $\tilde e$, there exists an edge
  $\tilde e'$ in $[\partial\pi(X');\partial\pi(Y')]$ such that $\tilde
  f(\tilde e')=\tilde p.\tilde e.\tilde s$ contains the edge $\tilde
  e$. Thus, the leaf $(X,Y)$ is in $\phi(C_{\tilde e'})$.

  We now proceed to prove that the union is a disjoint union. Let
  $(\tilde e',\tilde p,\tilde s)$ and $(\tilde e'',\tilde p',\tilde
  s')$ be two triples such that $\tilde f(\tilde e')=\tilde p.\tilde
  e.\tilde s$ and $\tilde f(\tilde e'')=\tilde p'.\tilde e.\tilde s'$.
  Assume that the intersection $\phi(C_{\tilde e'})\cap\phi(C_{\tilde
  e''})$ is non-empty. As $\partial^2\phi$ is a homeomorphism the
  intersection $C_{\tilde e'}\cap C_{\tilde e''}$ is non-empty and let
  $(X,Y)$ be a leaf in the intersection. As before, $(X,Y)$ is legal
  and $\tilde f$ restricts to a homeomorphism between the geometric
  realizations of $(X,Y)$ and its image. The edge $\tilde e$ is in
  both the images of $\tilde e'$ and $\tilde e''$ by this
  homeomorphism and thus $\tilde e'=\tilde e''$. It follows that the
  two tuples are equal and that the union in the Proposition is a
  disjoint union.

  Finally, let $\tilde f'$ be another lift of $f$ and let
  $\phi'\in\Phi$ be the automorphism associated to $\tilde f'$. There
  exists $u\in\FN$ such that $\tilde f'=u\tilde f$ and
  $\phi'=i_u\circ\phi$. For any edge $\tilde e'$ of $\tilde\Gamma$, let
  $\tilde e''={\phi'}\inv(u\inv)\tilde e'$. We have
\[
\tilde f'(\tilde e'')=\tilde f'({\phi'}\inv(u\inv)\tilde e')=\phi'({\phi'}\inv(u\inv))\tilde f'(\tilde e')=u\inv u \tilde f(\tilde e')=\tilde f(\tilde e')
\]
and
\[
\phi'(C_{\tilde e''})=\psi(C_{{\phi'}\inv(u\inv)}\tilde e')=\phi'({\phi'}\inv(u\inv)C_{\tilde e'})=u\inv\phi'(C_{\tilde e'})=\phi(C_{\tilde e'}).
\]
This proves that the decompositions obtained for $\tilde f$ and $\phi$
and for $\tilde f'$ and $\phi'$ are the same.
\end{proof}

\subsection{Prefix-suffix automaton}\label{sec:prefix}

We now use the previous Section to define the prefix-suffix automaton
for a train-track representative of an irreducible (with irreducible
powers) outer automorphism of the free group.  This automaton is a
classical tool in the case of substitutions, see~\cite{cs} and, indeed,
working with a substitution simplifies some technicalities. 

Let $\tau=(\Gamma,*,\pi,f)$ be a train-track representative of the
outer automorphism $\Phi$ of the free group $\FN$.  The
\textbf{prefix-suffix automaton} of $\tau$ is the finite oriented labelled graph
$\Sigma$ whose vertices are edges of $\Gamma$ and such that there is
an edge labelled by $(e',p,e,s)$ from $e$ to $e'$ if and only if the
reduced path $f(e')$ is equal to the reduced path $p.e.s$, where $p$ and $s$ are reduced paths in
$\Gamma$ (the \textbf{prefix} and the \textbf{suffix} respectively). We draw this edge as
\[
e\stackrel{p,s}{\longrightarrow}e'.
\]

An \textbf{$e$-path} $\sigma$ is a (finite or infinite) reduced path
in $\Sigma$ starting at the vertex $e$. The \textbf{length} $|\sigma|$
of an $e$-path $\sigma$ is its number of edges.  We denote by
$\Sigma_e$ the set of finite $e$-paths and by $\partial\Sigma_e$ the
set of infinite $e$-paths.

For a finite or infinite $e$-path $\sigma$ we denote by $\sigma(n)$
its $n$-th vertex (which is an edge of $\Gamma$). We write
$\sigma(0)=e$ and in particular $\sigma(|\sigma|)$ is the terminal
vertex of $\sigma$.

We now fix a lift $\tilde f$ of $f$ which is associated to the
automorphism $\phi\in\Phi$.

Let $\tilde e$ be an edge in the universal cover $\tilde\Gamma$ which
lies above the edge $e$ of $\Gamma$. For an edge
$e\stackrel{p,s}{\longrightarrow}e'$ of $\Sigma$ there exists a unique
edge $\tilde e'$ of $\tilde\Gamma$ which is a lift of $e'$ and such
that $\tilde f(\tilde e')=\tilde p\cdot\tilde e\cdot\tilde s$ where
$\tilde p$ and $\tilde s$ are lifts of $p$ and $s$ respectively.

Let $\sigma$ be an $e$-path and denote by
$e_{n-1}\stackrel{p_{n-1},s_{n-1}}{\longrightarrow}e_n$, its $n$-th
edge for $1\leq n\leq |\sigma|$ (with $e_0=e$).  By induction, for any
$1\leq n\leq |\sigma|$, there exists a unique edge $\tilde e_n$ of
$\tilde\Gamma$ which is a lift of $e_n=\sigma(n)$ and such that
$\tilde f(\tilde e_n)=\tilde p_{n-1}\cdot\tilde e_{n-1}\cdot\tilde
s_{n-1}$ where $\tilde p_{n-1}$ and $\tilde s_{n-1}$ are lifts of $p_{n-1}$ and
$s_{n-1}$ respectively (and $\tilde e_0=\tilde e$). We use the notation
\[
\sigma(\tilde e,\tilde f, n)=\tilde e_n.
\]

Let $\phi$ be the automorphism in the outer class $\Phi$ associated to
$\tilde f$.  For a finite $e$-path $\sigma$, we define
\[
C_{\tilde e,\sigma}=\phi^n(C_{\sigma(\tilde e,\tilde f,n)}).
\] 
We remark that this definition does not depend on the choice of the
automorphism $\phi$ in the outer class $\Phi$ and of the associated
lift $\tilde f$ of $f$.

Applying recursively Propostion~\ref{prop:lamfrac} we get

\begin{prop}\label{prop:declamiterated}
  Let $\Phi$ be an irreducible (with irreducible powers) automorphism of $\FN$ and $\Phi$ be its
  outer class.  Let $\tau=(\Gamma,*,\pi,f)$ be a train-track
  representative for $\Phi$.

  For any edge $\tilde e$ of $\tilde\Gamma$ and any $n\in\N$, 
\[
C_{\tilde e}=\biguplus_{\sigma\in\Sigma_{e}, |\sigma|=n} C_{\tilde e,\sigma} \qed
\]
\end{prop}

For an infinite $e$-path $\sigma$ we denote by
$C_{\tilde e,\sigma}$ the compact non-empty nested intersection
\[
C_{\tilde e,\sigma}=\bigcap_{\sigma'\mbox{\scriptsize prefix of }\sigma} C_{\tilde e,\sigma'}
\]
and we get

\begin{prop}\label{prop:rhoe}
  Let $\Phi$ be an irreducible (with irreducible powers) outer automorphism of $\FN$ and let
  $\tau=(\Gamma,*,\pi,f)$ be a train-track representative for
  $\Phi$. Let $\tilde e$ be a lift of an edge $e$ of $\Gamma$.

Then
\[
 C_{\tilde e}=\biguplus_{\sigma\in\partial\Sigma_{e}} C_{\tilde e,\sigma}. \qed
\]
\end{prop}

We denote by $\rho_{\tilde e}:C_{\tilde e}\to\partial\Sigma_{e}$ the
continuous onto map, which maps any leaf $(X,Y)$ in $C_{\tilde e}$ to
the unique infinite $e$-path $\sigma$ in $\partial\Sigma_{e}$ such
that $(X,Y)\in C_{\tilde e,\sigma}$.

The infinite $e$-path $\rho_{\tilde e}(X,Y)$ is the
\textbf{prefix-suffix representation} of the leaf $(X,Y)$ with respect
to its edge $\tilde e$.

Fixing a lift $\tilde f$ of $f$ and its associated automorphism
$\phi\in\Phi$, the action of $\phi$ on prefix-suffix representations
is easy to describe:

\begin{lem}\label{lem:phirho}
  Let $\tilde e_1$ be an edge of $\tilde\Gamma$. Let $(X,Y)$ be a leaf
  in $C_{\tilde e_1}$. Let $\tilde e_0$ be an edge of $\tilde\Gamma$
  such that $\tilde f(\tilde e_1)=\tilde
  p_0\cdot \tilde e_0\cdot\tilde s_0$ contains $\tilde e_0$.  Then
\[
\rho_{\tilde
  e_0}(\partial^2\phi(X,Y))=(e_0\stackrel{p_0,s_0}{\longrightarrow}e_1)\cdot\rho_{\tilde
  e_1}(X,Y)
\]
where $e_0,e_1,p_0$, and $s_0$ are the projections in $\Gamma$ of
$\tilde e_0,\tilde e_1,\tilde p_0$, and $\tilde s_0$ respectively.
\qed
\end{lem}

Roughly speaking this Lemma means that the action of $\phi$ on
prefix-suffix representations is by the shift map. This can be made
precise in the case of subsitutions. Indeed if $\sigma$ is a
substitution, in the basis $\CA$ of $\FN$, in the outer class $\Phi$
then the rose with $N$ petals is a a train-track representative for
$\Phi$. The universal cover $\tilde\Gamma$ is the Cayley graph of
$\FN$ and instead of the attracting lamination $L_\Phi$ we consider
the attracting subshift which consists of bi-infinite words in the
alphabet $\CA$. Such a bi-infinite word $Z$ encodes a bi-infinite
indexed path in $\tilde\Gamma$ that contains the origin. Thus $Z$
belongs to one of the cylinders $C_a$ where $a\in\CA$ is the letter at
index one in $Z$. Its prefix-suffix representation is computed with
respect to this cylinder.

These classical conventions in the case of subsitutions make the above
discussion on self-similarity of cylinders of the attracting
lamination and the description of the prefix-suffix automaton much
simpler.

\subsection{Prefix-suffix representation of periodic leaves}

In this section we continue our study of the prefix-suffix automaton.

\begin{prop}\label{prop:rhofinite}
Let $\Phi$ be an irreducible (with irreducible powers) outer automorphism of $\FN$ and let
  $\tau=(\Gamma,*,\pi,f)$ be a train-track representative for
  $\Phi$. Let $\tilde e_0$ be a lift of an edge $e_0$ of $\Gamma$ and
  let $\sigma$ be an infinite $e_0$-path in $\partial\Sigma_{e_0}$.

Then the compact set $C_{\tilde e_0,\sigma}$ is finite.
\end{prop}
\begin{proof}
  Fix an automorphism $\phi\in\Phi$ and an associated lift $\tilde f$
  of $f$ to $\tilde\Gamma$.  For each $n$, let $\tilde
  e_n=\sigma(\tilde e_0,\tilde f,n)$.

The length of the nested reduced path $\tilde f^n(\tilde e_n)$ goes to
infinity. There are two cases:

Either both extremities of the reduced path $\tilde f^n(\tilde e_n)$
goes to infinity in $\partial\tilde\Gamma$ and then $C_{\tilde
e,\sigma}$ contains only one element: the limit of
these paths.

Or, one of the extremities (say the initial one by symmetry) converges
to a vertex $\tilde v$ inside $\tilde\Gamma$. Let $\tilde v_n$ be the
initial vertex of $\tilde e_n$, for $n$ big enough, $\tilde f^n(\tilde
v_n)=\tilde v$.  As the length of the nested reduced path $\tilde
f^n(\tilde e_n)$ goes to infinity the terminal vertices of these paths
converge to a point $\partial\pi(Y)\in\partial\tilde\Gamma$.

For each $n$, let $E_n$ be the set of edges $\tilde e'_n$ in $\tilde
\Gamma$ such that $\tilde e'_n.\tilde e_n$ is a legal reduced path in
$\tilde\Gamma$. The cardinality of $E_n$ is bounded above by the
number of edges of $\Gamma$. For each $n$, and for each leaf $(X,Y)$
in $C_{\tilde e,\sigma}$, the leaf $\partial^2\phi^{-n}(X,Y)$ belongs
to $C_\Gamma(\tilde e_n)$ and thus to one of the $C_\Gamma(\tilde
e'_n\cdot\tilde e_n)$ for an $\tilde e'_n$ in $E_n$. Thus we can write
\[
C_{\tilde e,\sigma}\subseteq\bigcup_{\tilde e'_n\in E_n} C_\Gamma(\tilde f^n(\tilde
e'_n.\tilde e_n)).
\]
We can order each of the $E_n=\{\tilde e^1_n,\tilde
e^2_n,\ldots,\tilde e^{r_n}_n\}$ such that the reduced finite paths
$(\tilde f^n(\tilde e^k_n))_{n\in\N}$ are nested. For $n$ big enough
the terminal vertex of $\tilde f^n(\tilde e^k_n)$ is $\tilde v$ while
the lengths of these nested paths go to infinity and thus their
initial vertices converge to a point $X_k\in\partial\tilde\Gamma$. We
get that the sequence of nested paths $(\tilde f^n(\tilde e^k_n.\tilde
e_n))_{n\in\N}$ converges to a leaf $(X_k,Y)$ in $C_{\tilde
  e,\sigma}$.  This proves that the cardinality of $C_{\tilde
  e,\sigma}$ is bounded above by the number of edges of $\Gamma$.
\end{proof}

The prefix-suffix representations of periodic leaves of the attracting
lamination $\Latt$ have been described by Y.~Jullian in his PhD thesis
\cite{jull-these} where he obtains the following result. We give here
a proof because he only considers substitution automorphisms instead
of general train-tracks but this is only a technical and minor
improvement.

\begin{prop}[\cite{jull-these}]\label{prop:periodicleaves}
  Let $\Phi$ be an irreducible (with irreducible powers) outer automorphism of $\FN$ and let
  $\tau=(\Gamma,*,\pi,f)$ be a train-track representative for
  $\Phi$. Let $\phi$ be an automorphism in the outer class $\Phi$ and
  let $\tilde f$ be the associated lift of $f$ to $\tilde\Gamma$.  Let
  $\tilde e_0$ be a lift of an edge $e_0$ of $\Gamma$.  Let $(X,Y)$ be
  a leaf in $C_{\tilde e_0}$ and let $\sigma=\rho_{\tilde e_0}(X,Y)$
  be its prefix-suffix representation.

  The leaf $(X,Y)$ is periodic under the action of $\partial^2\phi$ if
  and only if its prefix-suffix representation $\sigma$ and the
  sequence $(\sigma(\tilde e_0,\tilde f, n))_{n\in\N}$ are
  pre-periodic.
\end{prop}
\begin{proof}
  Let $n$ be such that $\partial\phi^n(X)=X$ and $\partial\phi^n(Y)=Y$
  then $\tilde f^n$ restricts to an orientation preserving
  homeomorphism of the geometric realisation of the leaf $(X,Y)$. For
  each edge $e$ of $\Gamma$, the length of the path $f^k(e)$ increases
  to infinity with $k$. Thus either $\tilde f^n$ fixes a vertex of the
  bi-infinite path $[\partial\pi(X),\partial\pi(Y)]$ or there exists a
  unique edge $\tilde e_1$ of $[\partial\pi(X),\partial\pi(Y)]$ such
  that $\tilde e_1$ is a non-extremal edge of $\tilde f^n(\tilde
  e_1)$. In the first case we choose for $\tilde e_1$ the edge of
  $[\partial\pi(X),\partial\pi(Y)]$ which starts from the fixed vertex
  and lies in the same direction as $\tilde e_0$.  In both cases
  $\tilde f^n(\tilde e_1)$ contains $\tilde e_1$ and there exists
  $k_0$ such that $\tilde f^{nk_0}(\tilde e_1)$ contains $\tilde e_0$.

  Let $e_0$ and $e_1$ be the images of $\tilde e_0$ and $\tilde e_1$
  (respectively) in $\Gamma$.  Let $\sigma_0$ be the  finite
  $e_0$-path finishing at $e_1$ which corresponds to the fact that
  $\tilde e_0$ is an edge of $\tilde f^{nk_0}(\tilde e_1)$. Let
  $\sigma_1$ be the finite $e_1$-loop in $\Sigma_{e_1}$ which
  corresponds to the fact that $\tilde e_1$ is an edge of $\tilde
  f^{n}(\tilde e_1)$. Then
\[
\sigma=\rho_{\tilde e_0}(X,Y)=\sigma_0\cdot\sigma_1\cdot\sigma_1\cdot\sigma_1\cdots
\]
Moreover $\sigma_0(\tilde e_0,\tilde f, nk_0)=\tilde e_1$ and
$\sigma_1(\tilde e_1,\tilde f,n)=\tilde e_1$, which proves that the
sequence $(\sigma(\tilde e_0,\tilde f, n))_{n\in\N}$ is pre-periodic.

Conversely, assume that the prefix-suffix representation of the leaf
$(X,Y)$ is pre-periodic:
\[
\sigma=\rho_{\tilde e_0}(X,Y)=\sigma_0\cdot\sigma_1\cdot\sigma_1\cdot\sigma_1\cdots
\]
where $\sigma_0$ is a finite reduced path in $\Sigma_{e_0}$ finishing
at $e_1$ and $\sigma_1$ is a finite loop in $\Sigma_{e_1}$.  Assume
that $\sigma_0(\tilde e_0,\tilde f, |\sigma_0|)=\tilde e_1$ and
$\sigma_1(\tilde e_1,\tilde f,|\sigma_1|)=\tilde e_1$.

Then, applying Lemma~\ref{lem:phirho} we get that for all $n$
\[
\rho_{\tilde
e_1}(\partial^2\phi^{|\sigma_0|}(X,Y))=\sigma_1\cdot\sigma_1\cdot\sigma_1\cdots=\rho_{\tilde
e_1}(\partial^2\phi^{|\sigma_0|+n|\sigma|}(X,Y))
\]
From Proposition~\ref{prop:rhofinite}, the set $C_{\tilde e_1,\sigma'}$,
with
$\sigma'=\sigma_1\cdot\sigma_1\cdot\sigma_1\cdots$,
is finite and we get that there exists $m\neq n$ such that
\[
\partial^2\phi^{|\sigma_0|+m|\sigma_1|}(X,Y)=\partial^2\phi^{|\sigma_0|+n|\sigma_1|}(X,Y).
\]
This proves that $(X,Y)$ is periodic under the action of $\partial^2\phi$.
\end{proof}

\section{Repelling tree}\label{sec:tree}

We refer to K.~Vogtmann \cite{vogt-survey} for a survey and further
references on Outer Space and actions of the free group on $\R$-trees.

\subsection{Definition}\label{sec:repel-tree}

Let $\Phi$ be an irreducible (with irreducible powers) outer automorphism of $\FN$. The action of
$\Phi$ on the compactification of the projectivized Culler-Vogtman Outer
Space $\barCVN$ has exactly two fixed points $[T_\Phi]$ and $[T_{\Phi\inv}]$, one
attracting and one repelling. The action of $\Phi$ on $\barCVN$ has
North-South dynamic (see \cite{ll-north-south}).  The $\R$-trees $T_\Phi$
and $T_{\Phi\inv}$ have been described in \cite{gjll}.  The isometric actions
of $\FN$ on the $\R$-trees $T_\Phi$ and $T_{\Phi\inv}$ are both minimal, very
small and with dense orbits.

We will focus in this article on the repelling fixed point
$T_{\Phi\inv}$ of $\Phi$. We note that (the metric of) this tree is
only defined up to a multiplicative constant. But in this paper, we
pick-up a particular tree $T_{\Phi\inv}$ in the projective class
$[T_{\Phi\inv}]$

If we choose an automorphism $\phi$ in the outer class
$\Phi$, there exists a homothety, $H$ on $T_{\Phi\inv}$ which is
\textbf{associated} to $\phi$ (see \cite{gjll}):
\[
\forall P\in T_{\Phi\inv}, \forall u\in\FN, H(uP)=\varphi(u)H(P).
\]
The fixed point of the homothety $H$ may be in the metric completion $\bar
T_{\Phi\inv}$ rather than in $T_{\Phi\inv}$, and we regard $H$ as defined on this metric
completion.

With this convention, as $T_{\Phi\inv}$ is the repelling tree of $\Phi$, $H$ is
a contracting homothety of ratio
\[
\lambda=\frac{1}{\lambda_{\Phi\inv}}<1
\]
where $\lambda_{\Phi\inv}$ is the expansion factor of the irreducible (with irreducible powers) outer
automorphism $\Phi\inv$.

\subsection{The map $\CQ$}\label{subsec:Q}

Under the hypothesis of the previous section, $T_{\Phi\inv}$ is an
$\R$-tree with a minimal, very small action of $\FN$ by isometries
with dense orbits. We denote by $\hat T_{\Phi\inv}=\bar
T_{\Phi\inv}\cup \partial T_{\Phi\inv}$ the union of its metric
completion and of its Gromov boundary. The space $\hat T_{\Phi\inv}$
comes with the topology induced by the metric on
$T_{\Phi\inv}$. However, we consider the weaker \textbf{observers'
  topology} on $\hat T_{\Phi\inv}$. We refer to \cite{chl2} for
details on this topology. A basis of open sets is given by the
\textbf{directions}: a direction is a connected component of $\hat
T_{\Phi\inv}\smallsetminus\{P\}$ where $P$ is any point of $\hat
T_{\Phi\inv}$. We denote by $\TPhiinvobs$ the set $\hat T_{\Phi\inv}$
equipped with the observers' topology. The space $\TPhiinvobs$ is
Hausdorff, compact and has the same connected components than $\hat
T_{\Phi\inv}$. Indeed it is a dendrite in B.~Bowditch \cite{bowd-tree}
terminology.

\begin{thm}[\cite{chl2}]\label{thm:Qexists}
For any point $P\in \bar T_{\Phi\inv}$, the map $\CQ_P:\FN\to\TPhiinvobs
, u\mapsto uP$ has a unique equivariant continuous extension to a map
$\CQ:\partial\FN\to\TPhiinvobs$. This extension is independent of the
choice of the point $P$.
\end{thm}

The map $\CQ$ was first introduced in
\cite{ll-north-south,ll-periodic} with a slightly different approach.

Note that the map $\CQ$ fails to be continuous if we replace the
observers' topology by the stronger metric topology.

Let $P$ be a point in $\bar T_{\Phi\inv}$ and let $X$ be in
$\partial\FN$. Let $(u_n)_{n\in\N}$ be a sequence of elements of $\FN$
such that $u_n$ converges to $X$. For each $n$,
$H(u_nP)=\phi(u_n)H(P)$.  From Theorem~\ref{thm:Qexists}, and for the
observers' topology $u_nP$ converge towards $\CQ(X)$ while
$\phi(u_n)H(P)$ converge towards $\CQ(\partial\phi(X))$. Thus we have proved

\begin{lem}\label{lem:QH}
For any element $X\in\partial\FN$, $\CQ(\partial\phi(X))=H(\CQ(X))$.\qed
\end{lem}

\subsection{Dual lamination and the map $\CQ^2$}\label{subsec:Q2}

Using the map $\CQ$, in \cite{chl1-II}, a lamination $L(T_{\Phi\inv})$ dual to
the tree $T_{\Phi\inv}$ was defined.
\[
L(T_{\Phi\inv})=\{(X,Y)\in\partial^2\FN\ |\ \CQ(X)=\CQ(Y)\}.
\]
From this definition, the map $\CQ$ naturally induces an equivariant
map $\CQ^2: L(T_{\Phi\inv})\to \hat T_{\Phi\inv}$. It is proven in
\cite{chl1-II} that the map $\CQ^2$ is continuous (for the metric
topology on $\hat T_{\Phi\inv}$). The image $\Omega$ of $\CQ^2$ is the
\textbf{limit set} of $T_{\Phi\inv}$. It is contained in $\bar
T_{\Phi\inv}$ (equivalently, points of the boundary $\partial
T_{\Phi\inv}$ have exactly one pre-image by $\CQ$) but $\Omega$ may be
strictly smaller than $\bar T_{\Phi\inv}$, in particular it may fail
to be connected.

From Lemma~\ref{lem:QH}, the dual lamination $L(T_{\Phi\inv})$ is
invariant by $\Phi$ and we deduce

\begin{lem}\label{lem:Q2H}
  For any leaf $(X,Y)$ of the dual lamination $L(T_{\Phi\inv})$, we have
  $\CQ^2(\partial^2\phi(X,Y))=H(\CQ^2(X,Y))$.\qed
\end{lem}

\subsection{Dual and attracting laminations}\label{sec:dual-attr-lamin}

The dual lamination is sometime called the zero-length lamination and
it is clear to the experts that it contains the attracting
lamination. This is for example proven in \cite{hm-axes}.

\begin{prop}\label{prop:attracting}
  The attracting lamination $\Latt$ of an irreducible (with
  irreducible powers) outer automorphism $\Phi$ is a sublamination of
  the lamination $L(T_{\Phi\inv})$ dual to the repelling tree
  $T_{\Phi\inv}$ of $\Phi$:
\[
\Latt\subset L(T_{\Phi\inv}).
\]
\end{prop}
\begin{proof}
  Let $\tau=(\Gamma,*,\pi,f)$ be a train-track representative of
  $\Phi$.  Let $\tilde f$ be a lift of $f$ which is associated to the
  automorphism $\phi$ in the outer class $\Phi$. 

  Let $(X,Y)$ be a leaf in $\Latt$. Then by definition there exists an
  edge $e$ of the universal cover $\tilde\Gamma$ of $\Gamma$ and a
  sequence $u_n$ of elements of $\FN$ such that $u_n{\tilde f}^n(e)$
  converges to $(X,Y)$. Fix two base points in $\tilde\Gamma$ and
  $T_{\Phi\inv}$ (both denoted by $*$) and consider an equivariant map
  $q:\tilde\Gamma\to T_{\Phi\inv}$ such that $q(*)=*$ and which is affine on
  edges of $\tilde\Gamma$. Then for any vertex $P$ of $\tilde\Gamma$,
\[
q(\tilde f(P))=H(P).
\]
We deduce that the length of $q(u_n{\tilde f}^n(e))$ is
$\lambda^n$ times the length of $q(e)$ and as $\lambda<1$ this length
converges to $0$ when $n$ goes to infinity.

Let now $P_0$ be the start-point of $e$ and $P_1$ be its
end-point. Then $u_n{\tilde f}^n(P_0)$ converges to $\partial\pi(X)$
and $u_n{\tilde f}^n(P_1)$ converges to $\partial\pi(Y)$. The map
$\CQ$ is continuous for the weaker observers' topology on $\hat T$
(see \cite{chl2}), so that for this observers' topology $q(u_n{\tilde
f}^n(P_0))$ converges to $\CQ(X)$ and $q(u_n{\tilde f}^n(P_1))$
converges to $\CQ(Y)$.  The distance $d(q(u_n{\tilde f}^n(P_0)),
q(u_n{\tilde f}^n(P_1)))$ converges to $0$. The metric topology is
stronger than the observers' topology, thus the sequence $q(u_n{\tilde
f}^n(P_1))$ converges to $\CQ(X)$. As the observers' topology is
Hausdorff we conclude that $\CQ(X)=\CQ(Y)$. This proves that the leaf
$(X,Y)$ is in the dual lamination $L(T_{\Phi\inv})$ of $T_{\Phi\inv}$.
\end{proof}

\subsection{Self-similar structure}\label{sec:fractalomega}

Let $\Phi$ be an irreducible (with irreducible powers) outer automorphism of $\FN$. Let
$\tau=(\Gamma,*,\pi,f)$ be an irreducible train track representative for $\Phi$. Let
$\phi\in\Phi$ be an automorphism in the outer class $\Phi$ and $\tilde
f$ be the corresponding lift of $f$ to the universal cover
$\tilde\Gamma$ of $\Gamma$. 

Recall from Section~\ref{subsec:attlam} that for an edge $\tilde e$ of
$\tilde\Gamma$ we denote by $C_{\tilde e}$ the set of lines:
\[
C_{\tilde e}=C_\Gamma(\tilde e)\cap \Latt.
\]
Using the map $\CQ^2$ of Section~\ref{subsec:Q2} and
Proposition~\ref{prop:attracting}, we denote by $\Omega_{\tilde e}$ the subset
of $\bar T_{\Phi\inv}$:
\[
\Omega_{\tilde e}=\CQ^2(C_{\tilde e})=\CQ^2(C_\Gamma(\tilde e)\cap \Latt).
\]
As $\CQ^2$ is continuous, $\Omega_{\tilde e}$ is compact.

Using the irreducibility of the train-track $\tau$, eadge leaf of the
attracting lamination contains a translate of the edge $\tilde e$,
thus $\Latt=\FN.C_{\tilde e}$ and $\Omega=\FN.\Omega_{\tilde e}$.

Of course, the map $\CQ^2$ is invariant by the flip map. If $\tilde
e'$ is the reversed edge of $\tilde e$, the cylinders $C_{\tilde e}$
and $C_{\tilde e'}$ are homeomorphic by the flip map, and the
correponding sets of $\bar T_{\Phi\inv}$ are equal: $\Omega_{\tilde
  e}=\Omega_{\tilde e'}$.

We now apply $\CQ^2$ to Proposition~\ref{prop:lamfrac}.

\begin{prop}\label{prop:heartfrac}
  Let $\phi\in\Phi$ be an irreducible (with irreducible powers) automorphism of $\FN$ and $\Phi$ be its
  outer class.  Let $\tau=(\Gamma,*,\pi,f)$ be a train-track
  representative for $\Phi$ and $\tilde f$ a lift of $f$ to the
  universal cover $\tilde\Gamma$ associated to $\phi$.  

  For each edge $\tilde e$ which is a lift of the edge $e$ of $\Gamma$
\[
\Omega_{\tilde e}=\bigcup_{(\tilde e',\tilde p,\tilde s)} H(\Omega_{\tilde e'})
\]
where the finite union is taken over all triples $(\tilde e',\tilde
p,\tilde s)$ such that $\tilde e'$ is an edge of
$\tilde\Gamma$, $\tilde p.\tilde e.\tilde s$ is a reduced path in
$\tilde\Gamma$ and $\tilde f(\tilde e')=\tilde p.\tilde e.\tilde s$.

The above decomposition of $\Omega_{\tilde e}$ does not depend on the
choice of a particular automorphism $\phi$ in the outer class $\Phi$,
of the associated lift $\tilde f$ of $f$ and of the associated
homothety $H$.
\end{prop}
\begin{proof}
The equality follows directly by applying $\CQ^2$ to Proposition~\ref{prop:lamfrac}:
\begin{eqnarray*}
\Omega_{\tilde e}&=&\CQ^2(C_{\tilde e})\\
&=&\CQ^2(\uplus \phi(C_{\tilde e'}))\\
&=&\cup \CQ^2(\phi(C_{\tilde e'}))\\
&=&\cup H(\CQ^2(C_{\tilde e'}))\\
&=&\cup H(\Omega_{\tilde e'}) \qedhere
\end{eqnarray*}
\end{proof}

This self-similar structure of the compact subsets $\Omega_{\tilde e}$
takes place in the metric space $T_{\Phi\inv}$. Thus, this is exactly
that of a directed graph construction (see \cite{mw}) with similarity
ratios equal to the ratio $\frac{1}{\lambda_{\Phi\inv}}$ of the
homothety $H$.

But we lose the disjointness of the pieces in the self-similar
decomposition. Indeed, the $\Omega_{\tilde e'}$ involved in the
decomposion may fail to be disjoint. We will address this key issue
for the computation of the Hausdorff dimension in
section~\ref{sec:disjoint}.

\subsection{The maps $\CQ_{\tilde e}$}\label{sec:Qetilde}

Exactly as for cylinders of the attracting lamination, we can iterate
the decomposition. Let $\tilde e$ be an edge of $\tilde \Gamma$ that
is a lift of the edge $e$ of $\Gamma$. For any $e$-path
$\sigma$ in $\Sigma_e\cup\partial\Sigma_e$ we consider
\[
\Omega_{\tilde e,\sigma}=\CQ^2(C_{\tilde e,\sigma}).
\]
As above, using 
Propositions~\ref{prop:declamiterated} and
\ref{prop:rhoe} we get:

\begin{prop}\label{prop:omegasigma}
  Let $\Phi$ be an irreducible (with irreducible powers) automorphism of $\FN$ and $\Phi$ be its outer
  class.  Let $\tau=(\Gamma,*,\pi,f)$ be a train-track representative
  for $\Phi$.  Let $\tilde e$ be an edge of $\tilde\Gamma$. Let $\phi$
  be an automorphism in the outer class $\Phi$ and let $\tilde f$ be
  the associated lift of $f$. Let $H$ be the associated homothety of
  the attracting tree $T_{\Phi\inv}$ of $\Phi$.
\begin{enumerate}
\item For any $e$-path $\sigma$ of length $n$,
  $\Omega_{\tilde e,\sigma}=H^n(\Omega_{\sigma(\tilde e,\tilde f,n)})$
\item $\forall n\in\N\cup\{\infty\}$,\quad $\Omega_{\tilde
  e}=\displaystyle\bigcup_{\sigma\in\Sigma_{e}, |\sigma|=n}
\Omega_{\tilde e,\sigma}$
\item The map $\CQ^2$ factors through the map $\rho_{\tilde e}$: there
  exists a continuous map $\CQ_{\tilde e}:\partial\Sigma_{e}\to
  \Omega_{\tilde e}$ that makes the following diagram commutes:
\[
\xymatrix{
  C_{\tilde e}\ar[rr]^{\rho_{\tilde e}}\ar[rd]_{\CQ^2}&&\partial\Sigma_{e}\ar[ld]^{\CQ_{\tilde e}}\\
  &\Omega_{\tilde e} } \qed
\]
\end{enumerate}
\end{prop}

In the purpose of describing the self-similar structure of
$\Omega_{\tilde e}$ the choice of an orientation of each edge of
$\tilde\Gamma$ is irrelevant as the map $\CQ^2$ is
flip-invariant. Thus we could consider the smaller \textbf{unoriented
prefix-suffix automaton} $\Sigma^u$ which is obtained from the
prefix-suffix automaton $\Sigma$ by identifying two vertices $e_1$ and
$e_2$ if they are the same edge of $\Gamma$ with reverse orientations
and by identifying to edges
$e_1\stackrel{p_1,s_1}{\longrightarrow}{e'_1}$ and
$e_2\stackrel{p_2,s_2}{\longrightarrow}{e'_2}$ if the edges $e_1,e_2$
and $e'_1$ ,$e'_2$ are the same edge of $\Gamma$ with reverse
orientation and if $p_1,s_2$ and $s_1,p_2$ are the same paths in
$\Gamma$ with reverse orientations. 

In the classical context of substitutions the prefix-suffix automaton
has two symmetric connected components (one with positive letters and
one with inverses) and only the first one is usually considered.

\subsection{Attracting current}

A \textbf{current} for the free group $\FN$ is a Radon measure (recall
that a Radon measure is a Borel measure which is finite on compact
sets) on the double boundary $\partial^2\FN$ that is $\FN$-invariant
and flip-invariant.

As currents are $\FN$-invariant the action of the automorphism group
factors modulo inner automorphisms to a get an action of the outer
automorphism group $\Out(\FN)$ on the space of currents.

The irreducible (with irreducible powers) outer automorphism $\Phi$
has an attracting projectivized current $[\mu_\Phi]$ which was
introduced by R.~Martin \cite{mart}. Exactly as for the attracting
tree (and the repelling tree) we pick one current $\mu_\Phi$ in this
projetivized class.

This current satisfies
\[
\Phi.\mu_\Phi=\lambda_\Phi\mu_\Phi
\]
where $\lambda_\Phi$ is the expansion factor of $\Phi$.
That is to say, for every measurable set $A\subseteq\partial^2\FN$, 
\[
(\Phi.\mu_\Phi)(A)=\mu_\Phi(\phi\inv(A))=\lambda_\Phi\mu_\Phi(A)
\]
where $\phi$ is any automorphism in the outer class $\Phi$.

We refer to I.~Kapovich \cite{kapo-currents} for background, definitions and
statements on currents.

R.~Martin \cite{mart} proved that the support of $\mu_\Phi$ is
exactly the attracting lamination $\Latt$ of $\Phi$. It is proven in
\cite{chl1-III} that the lamination $L(T_{\Phi\inv})$, and thus its
sublamination $\Latt$ is uniquely ergodic.

This (projectivized) attracting current is better described if we use
the prefix suffix-automaton. Let $\tau=(\Gamma,*,\pi,f)$ be a
train-track representative of $\Phi$. Recall from
Section~\ref{sec:train-track-repr} that we denote by $(\mu_e)_e$ a
Perron-Frobenius eigen-vector of the transition matrix of $\tau$. From
the definition of $C_{\tilde e,\sigma}$ we get

\begin{lem}\label{lem:muphicylindre}
  For any edge $\tilde e$ of $\tilde\Gamma$ that lies above the edge $e$ of $\Gamma$
\[
\mu_\Phi(C_{\tilde e})=\mu_e.
\]
  Let $\sigma$ be a finite $e$-path in $\Sigma_e$ that ends at the
  edge $e'$ of $\Gamma$.  Then
\[
\mu_\Phi(C_{\tilde e,\sigma})=\frac{\mu_{e'}}{(\lambda_\Phi)^{|\sigma|}}.
\qed
\]
\end{lem}
(There is some fuzzyness in these equalities as both the eigen-vector
and the attracting current are only defined up to a mulitplicative
constant. The Lemma has to be understood as: there is a choice of
$\mu_\Phi$ and of $(\mu_e)_e$ such that...).

We consider $\nu_\Phi$ the push-forward of the attracting current
$\mu_\Phi$ by the continuous map $\CQ^2$ to the repelling tree
$T_{\Phi\inv}$:
That is to say for any measurable set $A$ in $\bar T_{\Phi\inv}$
\[
\nu_\Phi(A)=\mu_\Phi({\CQ^2}\inv(A)).
\]
From Lemma~\ref{lem:muphicylindre} we get

\begin{lem}\label{lem:nuphicylindre}
  For any edge $\tilde e$ of $\tilde\Gamma$ that lies above the edge $e$ of $\Gamma$
\[
\nu_\Phi(\Omega_{\tilde e})=2\mu_e.
\]
   Let $\sigma$ be a finite $e$-path in $\Sigma_e$ that ends at the
  edge $e'$ of $\Gamma$.  Then
\[
\nu_\Phi(\Omega_{\tilde e,\sigma})=2\frac{\mu_{e'}}{(\lambda_\Phi)^{|\sigma|}}.
\qed
\]
\end{lem}

The $2$ factor comes from the fact that we considered currents as
being invariant by the flip-map and that $\CQ^2$ is flip-invariant. As
both the attracting current and the metric of the repelling tree are
only defined up to a multiplicative constant this is totally
unsignificant.

\subsection{(Non-)Injectivity of $\CQ$}\label{sec:disjoint}

To get the Hausdorff dimension and measure of a self-similar metric
space a key feature is to know how much the self-similar pieces are
disjoint. In this purpose we collect results on the (non-)injectivity
of $\CQ$, $\CQ^2$ and $\CQ_{\tilde e}$ and we complete
Proposition~\ref{prop:heartfrac} by stating that the pieces in the
self-similar decomposition intersect in at most finitely many points.

Those results are much easier to state and prove in the case of
non-geometric outer automorphisms of the free group.  Recall that an
outer automorphism $\Phi$ of the free group is \textbf{geometric} if
it is induced by a homeomorphism $h$ of a surface $S$ with boundary
such that $\pi_1(S)=\FN$. In this case $h$ fixes the boundary
components of the fundamental group of $S$ and the action of $\FN$ of
the repelling and attracting trees $T_{\Phi\inv}$ and $T_\Phi$ are not
free. In this geometric case we have to deal with stabilizers of
points and fixed subgroups of the automorphisms in the outer class
$\Phi$. However the two trees $T_{\Phi\inv}$ and $T_\Phi$ are
\textbf{surface} (they are transverse to the lifts of the stable and
unstable foliations of $h$ on $S$), their limit sets $\Omega_{\tilde e}$
are intervals (or multi-interval) and the Hausdorff dimensions are $1$
which is not really striking.

On the opposite, if we assume that $\Phi$ is non-geometric then the
action of $\FN$ on the repelling and attracting trees are free and
automorphisms in the outer class $\Phi$ have trivial fixed
subgroups. This simplifies our work. Thus from now on we assume that
$\Phi$ is non-geometric.

The following result is proven in \cite{ch}. 

\begin{prop}[\cite{ch}]\label{prop:Qfinitetoone}
  Let $\Phi$ be an irreducible (with irreducible powers) non-geometric outer
  automorphism of $\FN$. Let $T_{\Phi\inv}$ be its repelling tree in
  the boundary of outer space.  

  Then $\CQ$ is finite-to-one and there are finitely many orbits of
  points in $\hat T_{\Phi\inv}$ with strictly more than two pre-images
  by $\CQ$.\qed
\end{prop}

From the definitions of $\CQ^2$ and, if we fix a train-track
representative $\tau=(\Gamma,*,\pi,f)$ and an edge $\tilde e$ of the
universal cover $\tilde\Gamma$, from the definition of $\CQ_{\tilde e}$,
we deduce

\begin{cor}\label{cor:Q2finitetoone}\label{cor:Qefinitetoone}
  $\CQ^2$ and $\CQ_{\tilde e}$ are finite-to-one.

  There are finitely many orbits of points in $\hat T_{\Phi\inv}$ with
  strictly more than two pre-images by $\CQ^2$ or with stricly more
  than one pre-image by $\CQ_{\tilde e}$.\qed
\end{cor}

From this corollary we can complete Proposition~\ref{prop:heartfrac}
by stating that the decomposition obtained there is not a partition
(as in Proposition~\ref{prop:lamfrac}) but nevertheless intersections
are finite.

\begin{prop}\label{prop:omegadisjoint}
  Let $\Phi$ be an irreducible (with irreducible powers) non-geometric outer
  automorphism of $\FN$. Let $T_{\Phi\inv}$ be its repelling tree in
  the boundary of outer space.  Let $\tau=(\Gamma,*,\pi,f)$ be a train-track representative for $\Phi$.

  Let $\tilde e$ be an edge of the universal cover $\tilde\Gamma$ lying above the edge $e$ of $\Gamma$. 
  Let $\sigma$ and $\sigma'$ be two distinct $e$-paths of length $n$. 

Then the intersection $\Omega_{\tilde e,\sigma}\cap\Omega_{\tilde e,\sigma'}$ is a finite set.
\end{prop}
\begin{proof}
  By Proposition~\ref{prop:lamfrac}, $C_{\tilde e,\sigma}$ and
  $C_{\tilde e,\sigma'}$ are disjoint.

  Assume by contradiction that there are infinitely many distinct
  elements $(P_n)_{n\in\N}$ in the intersection.  For each $n$, $P_n$
  has at least two pre-images by $\CQ_{\tilde e}$ (one starting by
  $\sigma$ and one starting by $\sigma'$). Applying
  Corollary~\ref{cor:Qefinitetoone}, up to passing to a subsequence,
  all the points $P_n$ are in the same orbit under the action of
  $\FN$: There exist elements $u_n\in\FN$ such that $P_n=u_nP_0$.

  From the commutative diagram in Proposition~\ref{prop:omegasigma},
  for each $n$, there exists elements $Z_n\in C_{\tilde e,\sigma}$ and
  $Z'_n\in C_{\tilde e,\sigma'}$ such that
  $\CQ^2(Z_n)=\CQ^2(Z'_n)=P_n$. As $\CQ^2$ is equivariant we get
\[
 \CQ^2({u_n}\inv Z_n)=\CQ^2({u_n}\inv Z'_n)=P_0,
\]
and as $\CQ^2$ is finite-to-one, up to passing to a subsequence we
assume that for all $n$ $Z_n=u_nZ_0$ and $Z'_n=u_nZ'_0$.

Again, up to passing to a subsequence we assume that the sequences
$(u_n)_{n\in\N}$ and $({u_n}\inv)_{n\in\N}$ converge to elements $U$
and $V$ respectively in $\partial\FN$ and, as $C_{\tilde e,\sigma}$
and $C_{\tilde e,\sigma'}$ are compact, that the sequences $(u_n
Z_0)_{n\in\N}$ and $({u_n Z'_0})_{n\in\N}$ converge to elements $Z$
and $Z'$.  We also assume that $(u_nV)_{n\in\N}$ converges to an
element $W\in\partial\FN$.

The action of $\FN$ on $\partial\FN$ is that of a convergence group, 
in particular, 
\[ 
\forall X\in\partial\FN\smallsetminus\{V\}\quad\lim_{n\to\infty} u_n X=
U.
\]
As the two ends of $Z$ (resp. $Z'$) are distinct, one of the two ends
of $Z_0$ (resp. $Z'_0$) is $V$. Thus $Z$ and $Z'$ are equal to $(U,W)$
or $(W,U)$.  As $C_{\tilde e,\sigma}$ and $C_{\tilde e,\sigma'}$ are
disjoint, $Z$ and $Z'$ have the same geometric realisation in reverse
order. But $C_{\tilde e}$ does not contain two paths in reverse
order. A contradiction.
\end{proof}

We now describe precisely the points with strictly more than one
pre-image by $\CQ_{\tilde e}$. For that we need to assume that the
outer automorphism $\Phi$ is forward rotationless. 

An irreducible (with irreducible powers), non-geometric, outer
automorphism $\Phi\in\partial\FN$ is \textbf{forward rotationless}
(see \cite{fh-recognition}) if for any integer $n$, for any
automorphism $\psi$ in the outer class $\Phi^n$ with strictly more
than two attracting fixed points in $\partial\FN$, there exists an
automorphism $\phi$ in the outer class $\Phi$ such that $\phi^n=\psi$
and such that each fixed point of $\psi$ is a fixed point of $\phi$.

From the following Proposition, we see that this extra hypothesis will
not restrict the scope of our results.

\begin{prop}[\cite{gjll}]\label{prop:rotationless}
There exists a constant $K_N$ depending only on $N$ such that for any  
irreducible (with irreducible powers), non-geometric outer automorphism $\Phi$,
the power $\Phi^{K_N}$ is forward rotationless.
\end{prop}

This Proposition is true for any outer automorphism but we restricted
ourself to the easier case of irreducible non-geometric automorphisms.

\begin{prop}\label{prop:singularperiodic}
Let $\Phi$ be an irreducible (with irreducible powers) non-geometric
forward rotationless outer automorphism of $\FN$. Let $T_{\Phi\inv}$
be its repelling tree in the boundary of outer space. Let
$\tau=(\Gamma,*,\pi,f)$ be a train track representative for $\Phi$ and
let $\tilde e$ be an edge of the universal cover $\tilde\Gamma$ of
$\Gamma$.

Let $P$ be a point in $\Omega_{\tilde e}$ with stricly more than one
pre-image by $\CQ_{\tilde e}$. Then any prefix-suffix representation
$\sigma\in\CQ_{\tilde e}\inv(P)$ is pre-periodic.

  Moreover, there exists a homothety $H$ of $T_{\Phi\inv}$ associated
  to an automorphism $\phi\in\Phi$ and to a lift $\tilde f$ of $f$
  such that $H(P)=P$ and for any pre-image $\sigma\in\CQ_{\tilde
  e}\inv(P)$ the sequence $(\sigma(\tilde e,\tilde f,n))_{n\in\N}$ is
  pre-periodic.
\end{prop}
\begin{proof}
  As $P$ has stricly more than one pre-image by $\CQ_{\tilde e}$, it
  has at least three different pre-images by $\CQ$. Let $\phi\in\Phi$
  be an automorphism in the outer class $\Phi$ and let $H$ be the
  associated homothety of $T_{\Phi\inv}$. For each integer $n$, by
  Lemma~\ref{lem:QH}, $\CQ\inv(H^n(P))=\phi^n(\CQ\inv(P))$ and by
  Proposition~\ref{prop:Qfinitetoone}, there exists an integer $n\geq 1$ and
  an element $u$ of $\FN$ such that $uH^n(P)=P$. 

  Let $\sigma$ be a pre-image by $\CQ_{\tilde e}$ of $P$. For any line
  $Z\in C_{\tilde e,\sigma}$, by Lemma~\ref{lem:Q2H}, and for any
  $k\in\N$, $(u(\partial^2\phi)^n)^k(Z)$ is also in the $\CQ^2$ fiber
  of $P$.  By Corollary~\ref{cor:Q2finitetoone}, $C_{\tilde
  e,\sigma}$ is finite and thus there exists $k\geq 1$ such that
\[
\forall Z\in C_{\tilde e}, \CQ^2(Z)=P \Rightarrow (u(\partial^2\phi)^n)^k(Z)=Z.
\]
Thus, elements $Z\in C_{\tilde e}$ in the $\CQ^2$ fiber of $P$ are
fixed points of the automorphism $\psi=(i_u\circ\phi^n)^k$ of the
outer class $\Phi^{nk}$. Moreover as $\CQ^2(Z)=P$, they are in the
attracting lamination $\Latt$ and thus there two ends are attracting
fixed points of $\psi$. As $\Phi$ was assumed to be forward
rotationless, there exists an automorphism $\phi'$ in the outer class
$\Phi$ that fixes all the elements $Z\in C_{\tilde e}$ such that
$\CQ^2(Z)=P$.  Let now $H'$ be the homothety associated to $\phi'$,
then $H'(P)=P$ and applying Proposition~\ref{prop:periodicleaves} we
proved the Proposition.
\end{proof}

A point $P$ in $\Omega_{\tilde e}$ is \textbf{first-singular}, if
there exists two distinct $e$-paths $\sigma$ and $\sigma'$ of length
$1$ such that $P\in\Omega_{\tilde e,\sigma}\cap\Omega_{\tilde
e,\sigma'}$. As there are finitely many $e$-path of length $1$, from
Proposition~\ref{prop:omegadisjoint} we get that there are finitely many first-singular
points in $\Omega_{\tilde e}$.

We prove the following technical result that we will use in the sequel.

\begin{lem}\label{lem:Qetildebound}
  Let $\tilde e$ be an edge of $\tilde\Gamma$ that lies above an edge
  $e$ of $\Gamma$. Let $n$ be an integer. Let $N$ be a set of $e$-paths of length $n$
  such that
\[
\forall\sigma,\sigma'\in N, \Omega_{\tilde e,\sigma}\cap\Omega_{\tilde e,\sigma'}\neq\emptyset.
\]
The cardinality of $N$ is bounded above by a constant $C_1$ depending only on $\Phi$.
\end{lem}
\begin{proof}
  Let $\sigma_0$ be the common prefix of all the elements of $N$. By
  self-similarity, we can replace $\Omega_{\tilde e}$ by
  $\Omega_{\tilde e,\sigma_0}$ and $N$ by the set $N'$ of suffixes
  $\sigma'$ of elements $\sigma=\sigma_0.\sigma'$ of $N$. Thus we
  assume that $\sigma_0$ has zero-length and that $N$ has strictly more than $1$ element.

  For each $\sigma$ in $N$, the set $\Omega_{\tilde e,\sigma}$
  contains a first-singular point $P$ and thus $\sigma$ is the prefix of
  length $n$ of one of the finitely many pre-images by $\CQ_{\tilde
    e}$ of $P$.
\end{proof}

\subsection{Hausdorff dimension and measure}

We refer to the book of K.~Falconer \cite{falc} for definitions of the
Hausdorff dimension and measure. 

For a metric space $(A,d)$, for any $\epsilon>0$, and $k>0$, let
\[
\mathcal{H}^k_\epsilon(A)=\inf \sum_{i\in\N} |A_i|^k\in\R^+\cup\{\infty\}
\]
where the infimum is taken over all coverings $(A_i)_{i\in\N}$ of $A$
such that the diameter $|A_i|$ of each $A_i$ is smaller than
$\epsilon$. If $A$ is compact, there are finite coverings of $A$ with
closed balls of diameter $\epsilon$, therefore
$\mathcal{H}^k_\epsilon(A)$ is finite.

For a homothety $H$ of ratio $\lambda$, one has
\[
\mathcal{H}^k_\epsilon(H(A))=\lambda^k\mathcal{H}^k_{\lambda\epsilon}(A).
\]

For a fixed $k>0$, $\epsilon\mapsto \mathcal{H}^k_\epsilon(A)$ is
decreasing and the \textbf{Hausdorff measure in dimension $k$} of $A$ is
\[
\mathcal{H}^k(A)= \lim_{\epsilon\to 0}\mathcal{H}^k_\epsilon(A)=
\sup_{\epsilon>0}\mathcal{H}^k_\epsilon(A)\in\R^+\cup\{\infty\}.
\]
Again, for a homothety $H$ of ratio $\lambda$, one has
\[
\mathcal{H}^k(H(A))=\lambda^k\mathcal{H}^k(A).
\]
 
The map $k\mapsto \mathcal{H}^k(A)$ is decreasing and takes values in
$\{0,\infty\}$ except in at most one point.  The \textbf{Hausdorff
dimension} of $A$ is
\[
\hdim(A)=\inf\{k\ |\ \mathcal{H}^k(A)=0\}=\sup\{k\ |\
\mathcal{H}^k(A)= \infty\}\in\R^+\cup\{\infty\}.
\]

From these definitions it is classical to deduce that the Hausdorff
dimension of a countable union $\cup_{i\in\N}X_i$ of subspaces of $A$
is the supremum of the dimensions of the $X_i$. In particular the
Hausdorff dimension of the limit set $\Omega$ is the maximum of the
Hausdorff dimension of the compact subsets $\Omega_{\tilde e}$, for all edges
$\tilde e$ of $\tilde \Gamma$.

\subsection{Main Theorem}\label{sec:mainthm}

The usual context to compute the
Hausdorff dimension of a self-similar set (or of a graph directed
construction) is inside $\R^n$ which is not the case here. Also, the
classical hypothesis to get the lower bound on the Hausdorff dimension
is by using the open set condition, of which we need to use a
non-classic version. We refer to \cite{mw} and \cite{edgar} for
computation of the Hausdorff dimension of graph directed constructions
and before them to the original article of J.~Hutchinson~\cite{hutch}
in the case of an iterated function system.

We are now ready to state and prove our main theorem.

\begin{thm}\label{thm:hausdim}
Let $\Phi$ be an irreducible (with irreducible powers) outer automorphism of the free group
$\FN$. Let $\tau=(\Gamma,*,\pi,f)$ be an irreducible train-track representative for
$\Phi$ and let $T_{\Phi\inv}$ be the repelling tree of $\Phi$.
  
The limit set $\Omega\subseteq\bar T_{\Phi\inv}$, and for each edge $\tilde e$ of $\tilde \Gamma$, the set
$\Omega_{\tilde e}\subseteq\bar T_{\Phi\inv}$ have Hausdorff dimension
\[
\delta=\hdim(\Omega)=\frac{\ln\lambda_\Phi}{\ln\lambda_{\Phi\inv}}
\]
where $\lambda_\Phi$ and $\lambda_{\Phi\inv}$ are the expansion
factors of $\Phi$ and $\Phi\inv$ respectively.
\end{thm}
\begin{proof}
  The limit set $\Omega$ is the union of translates by elements of
  $\FN$ of any $\Omega_{\tilde e}$. Thus the Hausdorff dimensions of
  these sets are all equal. 

  The repelling tree $T_{\Phi\inv}$ and the attracting lamination
  $L_\Phi$ do not change if we replace $\Phi$ by a power. Also, the
  expansion factor of a power $\Phi^n$ is $\lambda_{\Phi}^n$. Thus by
  Proposition~\ref{prop:rotationless}, up to replacing $\Phi$ by a
  suitable power, we assume that $\Phi$ is forward rotationless.

  For each edge $e$ of $\Gamma$ we choose one of its lifts in the
  universal cover $\tilde\Gamma$ and we denote by $\Omega_e$ the
  corresponding subset of $\bar T_{\Phi\inv}$. We denote by
  $E(\Gamma)$ the set of edges of $\Gamma$.

  From Proposition~\ref{prop:heartfrac}, we see that each of the
  pieces $\Omega_{e}$ is covered by finitely many translates of
  homothetic copies of the $\Omega_{e'}$. The number of copies
  used is given by the corresponding row in the transition matrix
  $M$ of the train-track $\tau$.  From the definition of the
  Hausdorff dimension it is straightforward to deduce for any
  $\epsilon>0$ and any $k>0$
\[
(\mathcal{H}^k_{\frac{\epsilon}{\lambda_{\Phi\inv}}}(\Omega_{e}))_{e
\in E(\Gamma)}\leq M(\mathcal{H}^k_{\frac{\epsilon}{\lambda_{\Phi\inv}}}(H(\Omega_{e'})))_{e'\in E(\Gamma)
}=\frac{1}{(\lambda_{\Phi\inv})^k}M(\mathcal{H}^k_\epsilon(\Omega_{e'}))_{e'
\in E(\Gamma)}
\]
where the comparison between these positive vectors is made
coordinatewise.

As the Perron-Frobenius eigen-value of $M$ is the expansion factor
$\lambda_\Phi$ of $\Phi$, we get by iteration that if
$k>\delta=\frac{\ln\lambda_\Phi}{\ln\lambda_{\Phi\inv}}$ for any edge
$\tilde e$ of $\tilde \Gamma$, $\mathcal{H}^k(\Omega_{\tilde e})=0$.
Also, if $k=\delta$ we get that
$(\mathcal{H}^\delta(\Omega_{e}))_{e\in E(\Gamma)}$ is bounded above
by the Perron-Frobenius eigen-vector of $M$. In particular, for any
edge $\tilde e$ of $\Gamma$, $\mathcal{H}^\delta(\Omega_{\tilde
  e})<\infty$.

This gives an upper bound for the Hausdorff dimension of each of the
$\Omega_{\tilde e}$ and an upper bound for the Hausdorff measure in dimension
$\delta$.

We now proceed to get the lower bound of the Hausdorff dimension and
measure. This involves describing quantitatively how much the maps
$\CQ_{\tilde e}$ fails to be injective and to evaluate how much they
contract the distances.

For two subsets $C$ and $C'$ of $\bar T_{\Phi\inv}$ we denote by
$\gap(C,C')$ the size of the gap between them:
\[
\gap(C,C')=\inf\{d(P,P')\ |\ P\in C, P'\in C'\}.
\]

We decompose the proof into three Lemmas.

\begin{lem}\label{lem:lowerbound}
  Let $\sigma$ and $\sigma'$ be two $e$-paths of
  length $n$ in $\Sigma_{e}$ such that $\Omega_{\tilde e,\sigma}\cap
  \Omega_{\tilde e,\sigma'}=\emptyset$. 

  There exists a constant $C_2>0$ depending only on $\Phi$ such that the
  gap between $\Omega_{ \tilde e,\sigma}$ and $\Omega_{\tilde e,\sigma'}$ is bigger
  than $\frac{C_2}{(\lambda_{\Phi\inv})^{n}}$:
\[
\gap(\Omega_{ \tilde e,\sigma},\Omega_{\tilde e,\sigma'})>\frac{C_2}{(\lambda_{\Phi\inv})^{n}}.
\]
\end{lem}
\begin{proof}
  By self-similarity, up to removing a common prefix to $\sigma$ and $\sigma'$
  and applying a homothety $H$, we assume that $\sigma$ and
  $\sigma'$ have different first edges $\sigma_1$ and
  $\sigma'_1$. 

  Let $0\leq p<n$ be the maximal length of prefixes $\sigma_p$ and
  $\sigma'_{p}$ of $\sigma$ and $\sigma'$ respectively such that
  $\Omega_{ \tilde e,\sigma_p}\cap\Omega_{\tilde
    e,\sigma'_p}\neq\emptyset$. As $\Omega_{ \tilde
    e,\sigma}\subseteq\Omega_{ \tilde e,\sigma_{p+1}}$ and $\Omega_{
    \tilde e,\sigma'}\subseteq\Omega_{ \tilde e,\sigma'_{p+1}}$ we get
  that
\[
\gap(\Omega_{ \tilde e,\sigma},\Omega_{\tilde e,\sigma'}) \geq
\gap(\Omega_{ \tilde e,\sigma_{p+1}},\Omega_{ \tilde
  e,\sigma'_{p+1}})>0.
\]
Thus, replacing $\sigma$ and $\sigma'$ by their prefixes of length
$p+1$, we assume that $p+1=n=|\sigma|=|\sigma'|$.

If $n=1$ (that is to say $p=0$) then there are only finitely many
choices of paths $\sigma$ and $\sigma'$ and $C_2$ has to be smaller than
the minimum of the gaps between all such possible choices of $\Omega_{
  \tilde e,\sigma}$ and $\Omega_{\tilde e,\sigma'}$.

Thus we assume that $n>1$ (and that $p=n-1>0$).

Let $P$ be a point in $\Omega_{\tilde e,\sigma_{n-1}}\cap
\Omega_{\tilde e,\sigma'_{n-1}}$. As $\sigma$ and $\sigma'$ does not have
common prefixes, $P$ is one of the finitely many first-singular points
in $\Omega_{\tilde e}$. Let $Z$ and $Z'$ be pre-images of $P$ by $\CQ^2$
in $C_{\tilde e,\sigma_{n-1}}$ and $C_{\tilde e,\sigma'_{n-1}}$
respectively. The point $P$ has at least two different pre-images by
$\CQ_{\tilde e}$, thus we can use Proposition~\ref{prop:singularperiodic} to
get that the pre-images by $\CQ_{\tilde e}$ of $Z$ and $Z'$ are
pre-periodic. $\sigma_{n-1}$ and $\sigma'_{n-1}$ are prefixes of two of
these pre-images. We also get a homothety $H$ of $T_{\Phi\inv}$ and an
associated automorphisms $\phi\in\Phi$ and lift $\tilde f$ of $f$ such
that the sequences $(\sigma_{n-1}(\tilde e,\tilde f,k))_{0\leq k\leq
  n-1}$ and $(\sigma'_{n-1}(\tilde e,\tilde f,k))_{0\leq k\leq n-1}$
only takes finitely many values. As the prefix-suffix automaton
$\Sigma$ is finite the terminal edges $\tilde e_n=\sigma(\tilde
e,\tilde f,n)$ and $\tilde e'_n=\sigma'(\tilde e,\tilde f,n)$ takes
only finitely possible values.

From our definitions
\[
\Omega_{ \tilde e,\sigma}=H^n(\Omega_{\tilde e_n})\mbox{ and }\Omega_{\tilde e,\sigma'}=H^n(\Omega_{\tilde e'_n}),
\]
thus, the lower bound of the gap is now given by:
\[
\gap(\Omega_{ \tilde e,\sigma},\Omega_{\tilde e,\sigma'}) = \frac
1{(\lambda_{\Phi\inv})^n}\gap(\Omega_{\tilde e_n},\Omega_{\tilde
  e'_n}).
\]
The existence of the constant $C_2$ follows from the finiteness of the
number of possible choices for $\tilde e_n$ and $\tilde e'_n$. 

For the sake of clarity, let us review this finiteness again. Using
the action of $\FN$ by isometries, we only need to consider one choice
of a lift $\tilde e$ of each edge $e$ of $\Gamma$. For each of these
$\tilde e$ we consider the finitely many first-singular points $P$ in
$\Omega_{\tilde e}$. For each of these first-singular points $P$ we
consider the associated lift $\tilde f$ of $f$ as given by
Proposition~\ref{prop:singularperiodic} and their finitely many
pre-periodic pre-images $\sigma$ by $\CQ_{\tilde
  e}$. Proposition~\ref{prop:singularperiodic} states that the sequence
$(\sigma(\tilde e,\tilde f,k))_{k\in\N}$ is contained in a finite set
$E_{\tilde e,P}$ of edges of $\tilde\Gamma$. Finally the edges $\tilde e_n$ and
$\tilde e'_n$ are among the finitely many edges of $\tilde\Gamma$ such
that $\tilde f(\tilde e_n)$ and $\tilde f(\tilde e'_n)$ contain one of
the edges of $E_{\tilde e,P}$.
\end{proof}

Let $P$ be a point in $\Omega_{\tilde e}$ and let $n$ be an
integer. We consider the following subset of $\Sigma_{e}$:
\[
N(P,n)=\{\sigma\in\Sigma_{e}\ |\ |\sigma|=n,
\gap(P,\Omega_{\tilde e,\sigma})\leq (\frac 1{\lambda_{\Phi\inv}})^n\}.
\]
The use of the set $N(P,n)$ is classical while proving lower bounds
for Hausdorff dimension and measure.

\begin{lem}\label{lem:NPn}
  There exists a constant $C_3$ depending only on the outer automorphism
  $\Phi$ such that for any point $P$ in $\Omega_{\tilde e}$ and any
  integer $n$, the set $N(P,n)$ has at most $C_3$ elements.
\end{lem}
\begin{proof}
  Let $k=\lceil\frac{\ln\frac
    2{C_2}}{\ln\lambda_{\Phi\inv}}\rceil$. For any elements
  $\sigma$ and $\sigma'$ in $N(P,n)$, by definition
  $\gap(\Omega_{\tilde e,\sigma},\Omega_{\tilde e,\sigma'})\leq
  \frac{2}{{\lambda_{\Phi\inv}}^n}$. We consider the prefixes,
  $\sigma_{n-k}$ and $\sigma'_{n-k}$ of $\sigma$ and
  $\sigma'$ of length $n-k$. The sets
  $\Omega_{\tilde e,\sigma_{n-k}}$ and $\Omega_{\tilde e,\sigma'_{n-k}}$
  contains the sets $\Omega_{\tilde e,\sigma}$ and
  $\Omega_{\tilde e,\sigma'}$, thus
\[
\gap(\Omega_{\tilde e,\sigma_{n-k}},\Omega_{\tilde e,\sigma'_{n-k}})\leq
\gap(\Omega_{\tilde e,\sigma},\Omega_{\tilde e,\sigma'})\leq
  \frac{2}{(\lambda_{\Phi\inv})^n}\leq \frac{C_2}{(\lambda_{\Phi\inv})^{n-k}}.
\]
From Lemma~\ref{lem:lowerbound} we get that these two sets are not disjoints:
\[
\Omega_{\tilde e,\sigma_{n-k}}\cap\Omega_{\tilde e,\sigma'_{n-k}}\neq\emptyset.
\]
From Lemma~\ref{lem:Qetildebound} the number of prefixes of length
$n-k$ of $N(P,n)$ is bounded above by $C_1$. Thus the cardinality of
$N(P,n)$ is bounded above by $C_3=C_1.(E)^k$, where $E$ is the number
of edges of the prefix-suffix automaton $\Sigma$.
\end{proof}

For a point $P$ in $\Omega_{\tilde e}$ and $r>0$ we denote by $B(P,r)$
the ball of radius $r$ in $\Omega_{\tilde e}$.

\begin{lem}\label{lem:nur}
There exists a constant $C_4$ depending only on $\Phi$ such that 
\[
\nu_\Phi(B(P,r))\leq C_4 r^\delta.
\]
\end{lem}
\begin{proof}
  Let $n=\lfloor\frac{\ln\frac 1r}{\ln\lambda_{\Phi\inv}}\rfloor$. For
  any point $Q$ in $B(P,r)$, let $\sigma$ be the prefix of length $n$
  of some pre-image by $\CQ_{\tilde e}$ of $Q$. Then $Q\in\Omega_{\tilde e,\sigma}$ and 
\[
\gap(P,\Omega_{\tilde e,\sigma})\leq d(P,Q)\leq r\leq (\frac 1{\lambda_{\Phi\inv}})^n.
\]
Thus we have proved that
\[
B(P,r)\subseteq\bigcup_{\sigma\in N(P,n)}\Omega_{\tilde e,\sigma}.
\]
Applying the push-forward $\nu_\Phi$ of the attracting current $\mu_\Phi$ we get
\[
\nu_\Phi(B(P,r))\leq\sum_{\sigma\in N(P,n)} \nu_\Phi(\Omega_{\tilde
  e,\sigma})
\]
For each $\sigma\in N(P,n)$
\[
\nu_\Phi(\Omega_{\tilde
  e,\sigma}) = \frac{\nu_\Phi(\Omega_{\sigma(\tilde e,\tilde f,n)})}{{\lambda_\Phi}^n}\leq  r^\delta\,\nu_\Phi(\Omega_{\sigma(\tilde e,\tilde f,n)}),
\]
and thus
\[
\nu_\Phi(B(P,r))\leq C_3\,r^\delta\,\max\{\nu_\Phi(\Omega_{\tilde e'})\ |\ \tilde e'\}
\]
which proves the Lemma.
\end{proof}

From Lemma~\ref{lem:nur} we deduce that the Hausdorff measure in
dimension $\delta$ is bounded below by the push forward of the
attracting current $\mu_\Phi$. This proves that the Hausdorff
dimension of $\Omega$ is bounded below by $\delta$.
\end{proof}

From the above proof we get that the Hausdorff measure in dimension
$\delta$ on the limit set $\Omega$ is not constant. Pulling back this
measure to the attracting lamination by the map $\CQ^2$ we get another
current supported by the attracting lamination $L_\Phi$. But we know
that the attracting lamination is uniquely ergodic thus we have
proved:

\begin{thm}\label{thm:measure}
The pushforward $\nu_\Phi$ of the attracting current $\mu_\Phi$ to the
limit set $\Omega$ is equal to the Hausdorff measure in dimension
$\delta$.\qed
\end{thm}

Once again this equality is to be understood up to a multiplicative
constant.

\subsection{Compact heart of trees}\label{sec:heart}

In this section we relate the sets $\Omega$ and $\Omega_{\tilde e}$ of
the previous section to the compact heart of $T_{\Phi\inv}$ as defined in
\cite{chl4}. 

We fix a basis $\CA$ of $\FN$. This is equivalent to fixing a chart
$(R_\CA,*,\pi)$ where $R_\CA$ is the rose with $N$ petals and $\pi$
the corresponding marking isomorphism. Elements $\partial\FN$ are
identified with infinite reduced words in $\CA^{\pm 1}$.

The unit-cylinder $C_\CA(1)$ of $\partial^2\FN$ is the set of lines
that goes through the origin, or equivalently, it is the set of pairs
of infinite reduced words $(X,Y)$ with distinct first letters. This is
a compact subset of $\partial^2\FN$ whose translates cover the double
boundary: $\FN.C_\CA(1)=\partial^2\FN$.

For a basis $\CA$ of $\FN$ the \textbf{compact limit set} of the
repelling tree $T_{\Phi\inv}$ is defined in \cite{chl4} by
\[
\Omega_\CA=\CQ^2(L(T_{\Phi\inv})\cap C_\CA(1)).
\]
It is a compact subset of $\bar T_{\Phi\inv}$.

From \cite{chkll} we know that the dual lamination $L(T_{\Phi\inv})$ is the
diagonal closure of the attracting lamination $\Latt$. In particular
\[
\Omega=\CQ^2(\Latt)=\CQ^2(L(T_{\Phi\inv}))
\]
\[
\Omega_\CA=\CQ^2(L(T_{\Phi\inv})\cap C_\CA(1))=\CQ^2(\Latt\cap
C_\CA(1)).
\]

The translates of the compact limit set $\Omega_\CA$ cover the limit
set $\Omega$: $\FN.\Omega_\CA=\Omega$.

As the Hausdorff dimension does not increase by taking countable
unions we get that the Hausdorff dimension of $\Omega_\CA$ is
$\delta$.

The compact heart, $K_\CA$ of $T_{\Phi\inv}$ is the convex hull
$\Omega_\CA$.  Recall that the tree $T_{\Phi\inv}$ can be covered by
countably many intervals (thus $T_{\Phi\inv}$ has Hausdorff dimension
$1$), and note that this is not the case of its metric completion
$\bar T_{\Phi\inv}$.  The compact heart $K_\CA$ is a subset of the
union $\Omega_\CA\cup T_{\Phi\inv}$. We get

\begin{thm}\label{thm:compactheart}
  Let $\Phi$ be an irreducible (with irreducible powers) outer
  automorphism of $\FN$. Let $\CA$ be a basis of $\FN$. Let
  $T_{\Phi\inv}$ be the repelling tree of $\Phi$. Let $\Omega_\CA$ be
  the compact limit set and $K_\CA$ be the compact heart of
  $T_{\Phi\inv}$ with respect to $\CA$. Then
\[
\hdim(\Omega_\CA)=\delta=\frac{\ln\lambda_\Phi}{\ln\lambda_{\Phi\inv}}\text{
  and }\hdim(K_\CA)=\max\{ 1,\
\frac{\ln\lambda_\Phi}{\ln\lambda_{\Phi\inv}}\}.
\qed
\]
\end{thm}

\section{Examples}\label{sec:examples}

In this section, we illustrate our result with two examples of irreducible (with irreducible powers) automorphisms. 

\subsection{Boshernitzan-Kornfeld example}

In \cite{bk} the following automorphism of $F_3$ is studied:
\[
\phi:\begin{array}[t]{rcl}a&\mapsto&b\\ b&\mapsto&caaa \\
c&\mapsto&caa\end{array}
\]
Let $\Phi$ be its outer class. We regard $\phi$ as a homeomorphism of
the rose with $3$ petals to get a train-track representative of $\Phi$
and the corresponding prefix-suffix automaton $\Sigma$, see
Figure~\ref{fig:bk-decomposition}.


\begin{figure}
\begin{center}
\input{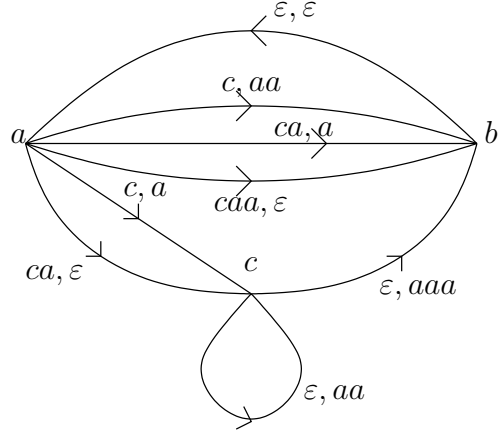}
\end{center}
\caption{\label{fig:bk-decomposition}Prefix-suffix automaton $\Sigma$ for Boshernitzan-Kornfeld automorphism}
\end{figure}

The transition matrix $M_\Phi$ and the expansion factor $\lambda_\Phi$
(which is the Perron-Frobenius eigen-value of $M_\Phi$) are
\[
M_\Phi=\left(\begin{array}{ccc}
0&3&2\\
1&0&0\\
0&1&1
      \end{array}\right)
\mbox{ and }
\lambda_{\Phi}
\approx 2.170.
\]

The inverse automorphism 
\[
\phi\inv:\begin{array}[t]{rcl}a&\mapsto&c\inv b\\ b&\mapsto&a \\
c&\mapsto&cb\inv cb\inv c\end{array}
\]
also defines on the rose with $3$ petals
 a train-track representative of $\Phi\inv$. The transition matrix and the expansion factor of $\Phi\inv$ are 
\[
M_{\Phi\inv}=\left(\begin{array}{ccc}
0&1&0\\
1&0&2\\
1&0&3
      \end{array}\right)
\mbox{ and }
\lambda_{\Phi\inv}
\approx 3.214.
\]
(note that positive and negative letters are both counted as one in the transition matrix).

M.~Boshernitzan and I.~Kornfeld associate this automorphism to an
interval translation map. 

Let $I=[0;1]$ be the unit interval and let $\alpha=\frac{1}{\lambda_{\Phi\inv}}\approx 0.311$ be
the positive root of $\alpha^3-\alpha^2-3\alpha+1=0$. The piecewise
translation $T:[0;1]\to [0;1]$ restricts to a translation on each of
the three intervals $I_a=[0;1-\alpha]$, $I_b=[1-\alpha;1-\alpha^2]$
and $I_c=[1-\alpha^2;1]$ of translation vector $\alpha$, $\alpha^2$
and $\alpha^2-1$ respectively.

\begin{figure}
\begin{center}
\input{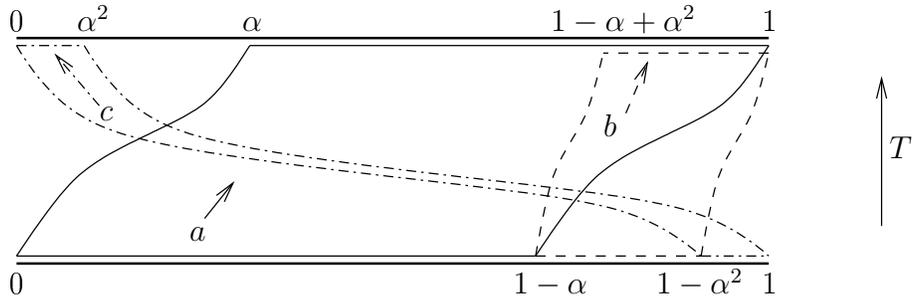}
\end{center}
\caption{\label{fig:bk}Interval translation of Boshernitzan-Kornfeld}
\end{figure}

The fact that $\phi$ and $T$ are associated can be seen by looking at
the first return map on the interval $[1-\alpha;1]$.

The repelling tree $T_{\Phi\inv}$ of $\Phi$ is the tree dual to (the lift to
the universal cover of) the vertical foliation of the mapping torus of
this interval translation. See \cite{gl-rank} and \cite{chl4} for a
precise construction of $T_{\Phi\inv}$ starting from the interval translation
$T$.

The compact heart $K_\CA$ of $T_{\Phi\inv}$ is the interval $I$. The
restriction of the action of the elements $a$, $b$ and $c$ of the
basis of $F_3$ to this interval $K_\CA=I$ are exactly the piecewise
exchange of the interval translation map, $T$.  The compact limit set
$\Omega_\CA$ is the limit set of the interval translation map:
\[
\Omega_\CA=\bigcap_{n\geq 0} T^n(I).
\]
This is a Cantor set
with Hausdorff dimension
$\frac{\ln\lambda_\Phi}{\ln\lambda_{\Phi\inv}}\approx 0,664$.

\subsection{Tribonacci example}

The following Tribonacci automorphism of $F_3$ has long been
studied
\[
\phi:\begin{array}[t]{rcl}a&\mapsto&ab\\ b&\mapsto&ac\\
c&\mapsto&a\end{array}
\]  
It is associated to what is known as {\em the} Rauzy fractal, and
X.~Bressaud studied its repelling tree (see \cite{bres}).

Let $\Phi$ be its outer class. We regard $\phi$ as a homeomorphism of
the rose with $3$ petals to get a train-track representative of $\Phi$
and the associated prefix-suffix automaton $\Sigma$, see
Figure~\ref{fig:tribo-decomposition}.

\begin{figure}
\begin{center}
\input{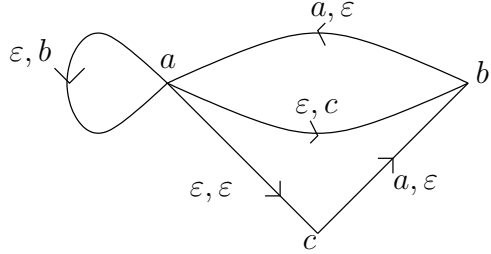}
\end{center}
\caption{\label{fig:tribo-decomposition}Prefix-suffix automaton $\Sigma$ for Tribonacci automorphism}
\end{figure}

The transition matrix and the
expansion factor are
\[
M_{\Phi}=\left(\begin{array}{ccc}
1&1&1\\
1&0&0\\
0&1&0
      \end{array}\right)
\mbox{ and }
\lambda_{\Phi}
\approx 1.839.
\]

A train-track representative of $\Phi\inv$ is given by the graph
$\Gamma$ of Figure~\ref{fig:tribo-inv-tt}, the homeomorphism $f$ and
the marking $\pi$:

\begin{figure}
\begin{center}
\input{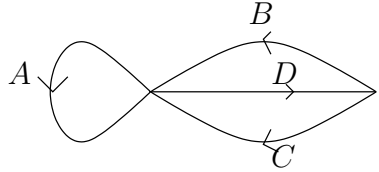} 
\end{center}
\caption{\label{fig:tribo-inv-tt}Graph $\Gamma$ of a train-track representative of the inverse of Tribonacci automorphism}
\end{figure}

\[
f:\begin{array}{rcl}A&\mapsto&DC\\ B&\mapsto&D\inv A\\ C&\mapsto&B\\ D&\mapsto&C\inv\end{array}
\quad
\pi:\begin{array}{rcl}a&\mapsto&A\\ b&\mapsto&DB\\ c&\mapsto&DC\end{array}
\]

The prefix-suffix automaton is given in Figure~\ref{fig:tribo-inv-pfa}.

\begin{figure}
\begin{center}
\input{tribo-inv-pfa.pstex_t} 
\end{center}
\caption{\label{fig:tribo-inv-pfa}Prefix-suffix automaton $\Sigma$ for the inverse of tribonacci automorphism}
\end{figure}

The outer automorphism $\Phi\inv$ has transition matrix and expansion factor
\[
M_{\Phi\inv}=
\left(\begin{array}{cccc} 
0&1&0&0\\
0&0&1&0\\
1&0&0&1\\
1&1&0&0\\
\end{array}\right)
\mbox{ and }
\lambda_{\Phi\inv}\approx 1.395. 
\]

The repelling tree $T_{\Phi\inv}$ has a connected limit set $\Omega_\CA=K_\CA$
which is of Hausdorff dimension $\delta=\frac{\ln\lambda_\Phi}{\ln\lambda_{\Phi\inv}}\approx 1.829$. The
$\R$-tree $K_\CA$, although compact, has Hausdorff dimension strictly bigger than 1.

X.~Bressaud has drawn nice pictures of (approximations of) the fractal
tree $K_\CA$ inside the Rauzy fractal (see \cite{bc}).

\bibliography{../CHL/bibli}{}

\newcommand{\etalchar}[1]{$^{#1}$}
\begin{thebibliography}{CHK{\etalchar{+}}08}

\bibitem[BC07]{bc}
Xavier Bressaud and Thierry Coulbois.
\newblock The {T}ree-{B}onacci.
\newblock Premier congrès franco-espagnol de mathématiques, Zaragoza, 2007.

\bibitem[BFH97]{bfh-lam}
Mladen Bestvina, Mark Feighn, and Michael Handel.
\newblock Laminations, trees, and irreducible automorphisms of free groups.
\newblock {\em Geom. Funct. Anal.}, 7(2):215--244, 1997.

\bibitem[BH92]{bh-traintrack}
Mladen Bestvina and Michael Handel.
\newblock Train tracks and automorphisms of free groups.
\newblock {\em Ann. of Math. (2)}, 135(1):1--51, 1992.

\bibitem[BK95]{bk}
Michael Boshernitzan and Isaac Kornfeld.
\newblock Interval translation mappings.
\newblock {\em Ergodic Theory Dynam. Systems}, 15(5):821--832, 1995.

\bibitem[Bow99]{bowd-tree}
Brian Bowditch.
\newblock {\em Treelike structures arising from continua and convergence
  groups}, volume 662.
\newblock Memoirs Amer. Math. Soc., 1999.

\bibitem[Bre07]{bres}
Xavier Bressaud.
\newblock Embedding a real tree in the {R}auzy fractal.
\newblock In Preparation, 2007.

\bibitem[CH08]{ch}
Thierry Coulbois and Arnaud Hilion.
\newblock {R}ips induction {I}: Index of a system of isometries.
\newblock In preparation, 2008.

\bibitem[CHK{\etalchar{+}}08]{chkll}
Thierry Coulbois, Arnaud Hilion, Ilya Kapovich, Gilbert Levitt, and Martin
  Lustig.
\newblock Attracting and dual laminations of iwip automorphisms.
\newblock Productive discussion, 2008.

\bibitem[CHL07]{chl2}
Thierry Coulbois, Arnaud Hilion, and Martin Lustig.
\newblock Non-unique ergodicity, observers' topology and the dual algebraic
  lamination for {$\mathbb{R}$}-trees.
\newblock {\em Illinois J. Math.}, 51(3):897--911, 2007.

\bibitem[CHL08a]{chl1-I}
Thierry Coulbois, Arnaud Hilion, and Martin Lustig.
\newblock {$\mathbb R$}-trees and laminations for free groups. {I}. {A}lgebraic
  laminations.
\newblock {\em J. Lond. Math. Soc. (2)}, 78(3):723--736, 2008.

\bibitem[CHL08b]{chl1-II}
Thierry Coulbois, Arnaud Hilion, and Martin Lustig.
\newblock {$\mathbb R$}-trees and laminations for free groups. {II}. {T}he dual
  lamination of an {$\mathbb R$}-tree.
\newblock {\em J. Lond. Math. Soc. (2)}, 78(3):737--754, 2008.

\bibitem[CHL08c]{chl1-III}
Thierry Coulbois, Arnaud Hilion, and Martin Lustig.
\newblock {$\mathbb R$}-trees and laminations for free groups. {III}.
  {C}urrents and dual {$\mathbb R$}-tree metrics.
\newblock {\em J. Lond. Math. Soc. (2)}, 78(3):755--766, 2008.

\bibitem[CHL09]{chl4}
Thierry Coulbois, Arnaud Hilion, and Martin Lustig.
\newblock {$\mathbb R$}-trees, dual laminations, and compact systems of partial
  isometries.
\newblock {\em Math. Proc. Cambridge Phil. Soc.}, 147:345--368, 2009.

\bibitem[CS01]{cs}
Vincent Canterini and Anne Siegel.
\newblock Geometric representation of substitutions of {P}isot type.
\newblock {\em Trans. Amer. Math. Soc.}, 353(12):5121--5144 (electronic), 2001.

\bibitem[Edg08]{edgar}
Gerald Edgar.
\newblock {\em Measure, topology, and fractal geometry}.
\newblock Undergraduate Texts in Mathematics. Springer, New York, second
  edition, 2008.

\bibitem[Fal90]{falc}
Kenneth Falconer.
\newblock {\em Fractal geometry}.
\newblock John Wiley \& Sons Ltd., Chichester, 1990.
\newblock Mathematical foundations and applications.

\bibitem[FH06]{fh-recognition}
Mark Feighn and Michael Handel.
\newblock The recognition theorem for {${\rm Out}(F\sb n)$}.
\newblock ArXiv:math/0612702, 2006.

\bibitem[FLP91]{flp}
A.~Fathi, F.~Laudenbach, and V.~Po\'{e}naru, editors.
\newblock {\em Travaux de {T}hurston sur les surfaces}.
\newblock Soci\'et\'e Math\'ematique de France, Paris, 1991.
\newblock S\'eminaire Orsay, Reprint of {\it Travaux de Thurston sur les
  surfaces}, Soc.\ Math.\ France, Paris, 1979 [ MR0568308 (82m:57003)],
  Ast\'erisque No. 66-67 (1991).

\bibitem[Fog02]{fogg}
N.~Pytheas Fogg.
\newblock {\em Substitutions in dynamics, arithmetics and combinatorics},
  volume 1794 of {\em Lecture Notes in Mathematics}.
\newblock Springer-Verlag, Berlin, 2002.
\newblock Edited by V.\ Berth\'e, S.\ Ferenczi, C.\ Mauduit and A.\ Siegel.

\bibitem[GJLL98]{gjll}
Damien Gaboriau, Andre Jaeger, Gilbert Levitt, and Martin Lustig.
\newblock An index for counting fixed points of automorphisms of free groups.
\newblock {\em Duke Math. J.}, 93(3):425--452, 1998.

\bibitem[GL95]{gl-rank}
Damien Gaboriau and Gilbert Levitt.
\newblock The rank of actions on {${\mathbb R}$}-trees.
\newblock {\em Ann. Sci. \'Ecole Norm. Sup. (4)}, 28(5):549--570, 1995.

\bibitem[Gui05]{guir-core}
Vincent Guirardel.
\newblock C\oe ur et nombre d'intersection pour les actions de groupes sur les
  arbres.
\newblock {\em Ann. Sci. \'Ecole Norm. Sup. (4)}, 38(6):847--888, 2005.

\bibitem[HM06]{hm-axes}
Michael Handel and Lee Mosher.
\newblock Axes in outer space. 2006.
\newblock arXiv:math/0605355.

\bibitem[Hut81]{hutch}
John~E. Hutchinson.
\newblock Fractals and self-similarity.
\newblock {\em Indiana Univ. Math. J.}, 30(5):713--747, 1981.

\bibitem[Jul09]{jull-these}
Yann Jullian.
\newblock {\em Représentations géométriques des systèmes dynamiques
  substitutifs par substitutions d'arbre}.
\newblock PhD thesis, Universit\'e Aix-Marseille~II, 2009.

\bibitem[Kap06]{kapo-currents}
Ilya Kapovich.
\newblock Currents on free groups.
\newblock In {\em Topological and asymptotic aspects of group theory}, volume
  394 of {\em Contemp. Math.}, pages 149--176. Amer. Math. Soc., Providence,
  RI, 2006.

\bibitem[LL03]{ll-north-south}
Gilbert Levitt and Martin Lustig.
\newblock Irreducible automorphisms of {$F\sb n$} have north-south dynamics on
  compactified outer space.
\newblock {\em J. Inst. Math. Jussieu}, 2(1):59--72, 2003.

\bibitem[LL08]{ll-periodic}
Gilbert Levitt and Martin Lustig.
\newblock Automorphisms of free groups have asymptotically periodic dynamics.
\newblock {\em J. Reine Angew. Math.}, 619:1--36, 2008.

\bibitem[Mar95]{mart}
Reiner Martin.
\newblock {\em Non-Uniquely Ergodic Foliations of Thin Type, Measured Currents
  and Automorphisms of Free Groups}.
\newblock PhD thesis, UCLA, 1995.

\bibitem[MW88]{mw}
R.~Daniel Mauldin and S.~C. Williams.
\newblock Hausdorff dimension in graph directed constructions.
\newblock {\em Trans. Amer. Math. Soc.}, 309(2):811--829, 1988.

\bibitem[Vog02]{vogt-survey}
Karen Vogtmann.
\newblock Automorphisms of free groups and outer space.
\newblock {\em Geom. Dedicata}, 94:1--31, 2002.

\end{thebibliography}
\bibliographystyle{alpha}

\end{document}